\documentclass[11pt]{article}
\usepackage{amsmath,amsfonts,amssymb, amsthm,enumerate,graphicx}
\usepackage{tikz,authblk}

\usepackage{wrapfig}
\usepackage{subcaption}
\usepackage{url}
\usepackage[colorlinks=true, allcolors=blue]{hyperref}
\usepackage{stackengine,scalerel}
\usetikzlibrary {decorations.pathmorphing, decorations.pathreplacing, decorations.shapes}

\newtheorem{theorem} {{\textsf{Theorem}}}
\newtheorem{proposition}[theorem]{{\textsf{Proposition}}}

\newtheorem{definition}[theorem]{{\textsf{Definition}}}

\newtheorem{problem}[theorem]{{\textsf{Problem}}}
\newtheorem{lemma}[theorem]{{\textsf{Lemma}}}

\begin{document}

\title{The Homotopy Types of the Independence and Perfect Matching Complexes of M\"obius and Circular Ladder Graphs}

\author{Anshu Agarwal and Biplab Basak$^1$}
	
\date{September 14, 2026}
	
\maketitle
	
\vspace{-10mm}
\begin{center}
		
\noindent {\small Department of Mathematics, Indian Institute of Technology Delhi, New Delhi, 110016, India.$^2$}

\footnotetext[1]{Corresponding author}
		
\footnotetext[2]{{\em E-mail addresses:} \url{maz228084@maths.iitd.ac.in} (A. Agarwal), \url{biplab@iitd.ac.in} (B. Basak).}
		
\medskip
	
\end{center}
	
\hrule
	
\begin{abstract}
The independence complex and perfect matching complex of a graph are simplicial
complexes encoding, respectively, its independent sets and perfect matchings.
Determining their homotopy types is generally difficult, with explicit
descriptions known mainly for highly structured graph families. In this
article, we determine the homotopy types of these complexes for the
M\"obius ladder graphs $M_{2n}$ and circular ladder graphs
$\mathcal{C}_{2n}$. The M\"obius ladder graphs $M_{2n}$ are highly symmetric cubic graphs
obtained from a $2n$-cycle by joining opposite vertices, while the circular ladder graphs $\mathcal{C}_{2n}$ are the Cartesian products of an $n$-cycle and a path of length one. We show that $\operatorname{Ind}(M_{2n})$ and
$\operatorname{Ind}(\mathcal{C}_{2n})$ have the homotopy type of wedges of
spheres, with the numbers and dimensions of the spheres exhibiting periodic behavior according to $n$ modulo $4$. For the perfect matching complex
$\mathcal{M}_p(M_{2n})$, its homotopy type is a wedge of two copies of
$\mathbb{S}^{(n-2)/2}$ when $n$ is even, while for odd $n$ it has the
homotopy type of a wedge of spheres whose numbers and dimensions depend
periodically on $n$ modulo $6$. The perfect matching complex
$\mathcal{M}_p(\mathcal{C}_{2n})$ is contractible for odd $n$, whereas for
even $n$ its homotopy type is a wedge of spheres, with the numbers and
dimensions determined periodically by $n$ modulo $6$. Thus, we obtain
explicit homotopy types for the independence and perfect matching
complexes of two highly symmetric families of cubic graphs, which are also
relevant in crystallization theory and the combinatorial representation of
PL manifolds.
\end{abstract}

\noindent {\small {\em MSC 2020\,:} 05C69, 05C70, 05E45, 55P15, 55U10.
		
}
	
\medskip

\section{Introduction}
Let $G=(V,E)$ be a finite simple graph. The \emph{independence complex} of $G$, denoted by $\operatorname{Ind}(G)$, is the simplicial complex whose faces are the independent sets of $G$, that is, the subsets of $V$ containing no two adjacent vertices. Despite this elementary definition, independence complexes often exhibit remarkably rich and subtle topological behaviour. They provide a natural bridge between graph theory and topology and have become important objects of study in topological and enumerative combinatorics.

The importance of independence complexes is further emphasized by their close relationship with other simplicial complexes associated with graphs. For example, the matching complex of a graph can be realized as the independence complex of its line graph. Consequently, several developments in the study of matching complexes draw upon methods and results originating in the theory of independence complexes; see, for instance, \cite{BH17,Matsushitamatching}. More generally, the study of topological properties of simplicial complexes arising from graph-theoretic structures has developed into an important theme in topological combinatorics. Jonsson's monograph~\cite{J08} provides a systematic account of this subject and discusses a wide range of simplicial complexes naturally associated with graphs and related combinatorial structures.

Over the years, numerous methods have been developed to investigate the topology of independence complexes. Among them are the matching tree algorithm, the star cluster method, the fold lemma, and systematic applications of deletion and link operations; see, for example, \cite{a12, bousquet2008independence,engstrom, e08, kawmurahomotopytype,Matsushitantimes4,Matsushitantimes6}. These techniques have proved particularly useful in determining homotopy types and homological properties of independence complexes associated with structured graph families. In addition, Jonsson~\cite{J08} developed a broader matroid-theoretic framework for complexes of forests, generalizing independence complexes of matroids. Star clusters in independence complexes of graphs were studied in \cite{Barmakstarclusters}. The independence complexes of path graphs and cycle graphs were studied in \cite{k99}.
 The total Betti number of the independence complexes of ternary graphs was investigated in \cite{Zhsng25}. The homotopy type of the independence complex of graphs with no induced cycles of length divisible by \(3\) was studied in \cite{Kim2022}. Further results concerning the homotopy types of independence complexes can be found in \cite{cs26, Samir21}.

Another important simplicial complex associated with a graph is its \emph{matching complex}, whose vertices are the edges of $G$ and whose faces are the matchings of $G$. Matching complexes have been extensively studied in topology, algebra, and combinatorics, beginning with the work of Bouc~\cite{bouc1992homologie}; a prominent family is given by the matching complexes of complete bipartite graphs, known as \emph{chessboard complexes}~\cite{christos,bjorner1994chessboard,jojic2018h,shareshian2007torsion,wachs2003topology, ziegler1994shellability}. Matching complexes for other families of structured graphs have also been studied; see \cite{GeneralBayer,matsushita2022matching,Matsushitamatching}.
 A related complex is the \emph{perfect matching complex} $\mathcal{M}_p(G)$, defined, for a graph $G$ admitting a perfect matching, as the pure simplicial complex whose facets are the perfect matchings of $G$; thus, it records the combinatorial interactions among perfect matchings and forms a subcomplex of the matching complex. The topology of perfect matching complexes has recently received increasing attention; for example, Bayer et al.~\cite{bayer2022perfect} studied their homotopy types for certain honeycomb graphs, while perfect matching complexes of polygonal line tilings were studied in \cite{cs25}. In \cite{McSorley}, various combinatorial structures associated with the M\"obius ladder graph have been studied.

In this article, we study the topology of the independence complexes and perfect matching complexes of M\"obius ladder graphs and circular ladder graphs. Both families form natural and highly symmetric classes of cubic graphs and provide interesting settings for studying graph-associated simplicial complexes. M\"obius ladder graphs are obtained from an even cycle by adding edges joining pairs of antipodal vertices (see Section \ref{defi}). The cyclic structure, together with these additional antipodal edges, imposes nontrivial restrictions on independent sets, while, at the same time, their high degree of symmetry gives rise to a rich collection of perfect matchings. These features make M\"obius ladder graphs a natural class of graphs for investigating the topology of independence complexes and perfect matching complexes. Circular ladder graphs have a similar cyclic structure and can be viewed as the Cartesian product of a cycle with a path of length one (see Section \ref{defi}). They are also $3$-regular and possess a high degree of symmetry. Their structure provides a natural framework for studying independent sets and perfect matchings and, consequently, the topology of the associated independence and perfect matching complexes. In this article, we investigate these two families from a topological perspective and compare the resulting homotopy types.

These graphs are of particular interest from the viewpoint of the interplay between graph theory and topology. Indeed, since they are cubic graphs admitting a proper edge-coloring with three colors, say ${0,1,2}$, they naturally give rise to $3$-colored simplicial cell complexes. To see this, associate to each vertex $v$ of the graph a triangle whose vertices are labeled by the colors $0,1,2$. If two vertices $u$ and $v$ are joined by an edge of color $i$, then identify the edge of the triangle corresponding to $u$ with the edge of the triangle corresponding to $v$ opposite the vertex labeled $i$. Thus, the edge-coloring determines the gluing pattern of the triangles and, consequently, produces a closed connected surface. This construction provides a natural combinatorial bridge between graph theory and topology. The vertices and colored edges of the graph encode the triangles and their gluings, respectively, while the resulting surface captures topological information inherent in the combinatorial structure of the graph. In particular, this construction is a basic instance of the colored-graph approach to topology and illustrates how properly edge-colored graphs can serve as combinatorial representations of closed surfaces. For further details on crystallization theory, we refer the reader to \cite{b17, fgg86}.

The independence complexes $\operatorname{Ind}(M_{2n})$ and $\operatorname{Ind}(\mathcal{C}_{2n})$, together with the perfect matching complexes $\mathcal{M}_p(M_{2n})$ and $\mathcal{M}_p(\mathcal{C}_{2n})$, exhibit rich and intricate combinatorial structures that are reflected in their underlying topological properties. In this article, we determine their homotopy types and prove the following results.

\bigskip

\noindent {\bf Main Theorem 1.} {\em Let $M_{2n}$ be the the M\"obius ladder graph, $n\geq 2$. Then we have the following.}

\begin{enumerate}[$(a)$]
    \item $ \operatorname{ Ind}(M_{2n})\simeq \begin{cases}
       \mathbb{S}^{2k-1} & \text{\em if } ~  n=4k, \\ \bigvee_3 \ \mathbb{S}^{2k} & \text{\em if } ~ n=4k+2, \\
       \mathbb{S}^{2k} & \text{\em if} ~ n=4k+1, 4k+3,
       
   \end{cases}$

   \item $\mathcal M_p(M_{2n}) \simeq \bigvee_2 \ \mathbb{S}^{k-1}$  {\em if} $n=2k$,

   \item $ \mathcal M_p(M_{2n})\simeq \begin{cases}
       \bigvee_4 \ \mathbb{S}^{2k+1} &  \text{\em if }  ~  n=3(2k+1), \\ \bigvee_2 \ \mathbb{S}^{2k+2} & \text{\em if } ~  ~ n=3(2k+1)+2, 3(2k+1)+4.
   \end{cases}$
\end{enumerate}

\bigskip

\noindent {\bf Main Theorem 2.} {\em Let $\mathcal{C}_{2n}$ be the the circular ladder graph, $n\geq 2$. Then we have the following.}

\begin{enumerate}[$(a)$]
    \item $ \operatorname{ Ind}(\mathcal C_{2n})\simeq \begin{cases}
    \bigvee_3 \mathbb{S}^{2k-1} & \text{if } n=4k, \\   \mathbb{S}^{2k-1} & \text{if } n=4k+1 \\ \mathbb{S}^{2k} & \text{if } n=4k+2,\\
    \mathbb{S}^{2k+1} & \text{if } n=4k+3,
   \end{cases}$

   \item $\mathcal M_p(\mathcal C_{2n}) \simeq \ast$  {\em if} $n=2k+1$,

   \item $ \mathcal M_p(\mathcal C_{2n})\simeq \begin{cases}
       \mathbb{S}^{\frac{n-2}{2}}\ \vee \ \mathbb{S}^{\frac{n-2}{2}} \bigvee_4 \ \mathbb{S}^{2k} & \text{if } n=3(2k), \\ \mathbb{S}^{\frac{n-2}{2}}\ \vee \ \mathbb{S}^{\frac{n-2}{2}} \bigvee_2 \ \mathbb{S}^{2k+1} & \text{if } n=3(2k+1)\pm 1.
   \end{cases}$
\end{enumerate}

\section{Preliminaries}\label{pre}
A \emph{graph} \(G\) is an ordered pair \(G=(V,E)\), where \(V\) is a finite set whose elements are called \emph{vertices}, and \(E\) is a set of pairs \(uv\), where \(u,v\in V\) and \(u\neq v\). The elements of \(E\) are called \emph{edges}. If \(uv\in E\), then \(u\) and \(v\) are said to be \emph{adjacent}. We define the \textit{open neighborhood} of a vertex $v$ in $G$, denoted by $N_G(v)$, as $N_G(v)=\left\{u\in V(G)\;\middle|\;u\neq v\text{ and }uv\in E(G)\right\}$. Similarly, the \textit{closed neighborhood} of $v$ in $G$, denoted by $N_G[v]$, is defined as $N_G[v]=N_G(v)\sqcup\{v\}$. When the underlying graph is clear from the context, we omit the subscript $G$ and write $N(v)$ and $N[v]$ to denote the open and closed neighborhoods of $v$, respectively. Let $S$ be a subset of $V(G)$. The \textit{induced subgraph} of $G$ on $S$, denoted by $G[S]$, is the graph whose vertex set is $S$ and whose edge set consists of all edges of $G$ having both endpoints in $S$. For a subset $H\subseteq V(G)$, we denote the induced subgraph $G[V(G)\setminus H]$ by $G\setminus H$. When $H$ is a singleton set, say $\{v\}$, we denote the induced subgraph $G[V(G)\setminus \{v\}]$ by $G\setminus v$.

A \emph{simplicial complex} \(K\) on a vertex set \(V\) is a collection of finite subsets of \(V\), called \emph{faces} or \emph{simplices}, such that if \(\sigma\in K\) and \(\tau\subseteq \sigma\), then \(\tau\in K\). In particular, \(K\) contains the empty set. A \emph{facet} of a simplicial complex \(K\) is a face of \(K\) that is maximal with respect to inclusion. That is, a face \(\sigma\in K\) is a facet if there is no face \(\tau\in K\) such that $\sigma\subsetneq \tau$. Let \(K\) be a simplicial complex with vertex set \(V(K)\) of cardinality $m$. The \emph{geometric realization} (or \emph{geometric carrier}) of \(K\), denoted by \(\lvert K\rvert\), is the topological space obtained by realizing each simplex of \(K\) as a geometric simplex and gluing these simplices along their common faces according to the face relations in \(K\). More precisely, $$\lvert K\rvert=\left\{(\lambda_v)_{v\in V(K)}\in \mathbb{R}^{m}\mid \lambda_v\geq 0,\ \sum_{v\in V(K)}\lambda_v=1,\ \{v\in V(K):\lambda_v>0\}\in K\right\},$$ endowed with the subspace topology inherited from \(\mathbb{R}^{m}\).
We refer to \cite{J08} for fundamental terms and facts concerning simplicial complexes. Let $v\in V(K)$. The \emph{link} of $v$ in $K$ is defined as
$\{\sigma\in K : v\notin\sigma,\ \sigma\cup\{v\}\in K\}$,
while the \emph{deletion} of $v$ from $K$ is defined as
$\{\sigma\in K : v\notin\sigma\}$.

Let \(K\) and \(L\) be two simplicial complexes with disjoint vertex sets. The \emph{join} of \(K\) and \(L\), denoted by \(K\star L\), is the simplicial complex defined by
\[
K\star L=\{\sigma\cup\tau : \sigma\in K,\ \tau\in L\}.
\]
Equivalently, the faces of \(K \star L\) are precisely the unions of a face of \(K\) and a face of \(L\). Let \(K\) be a simplicial complex and let \(S^0=\{u,v\}\) be the simplicial complex consisting of two isolated vertices. The \emph{suspension} of \(K\), denoted by \(\Sigma K\), is defined as the join \(K\star S^0\), that is, \(\Sigma K=K\star\{u,v\}\). Let \(K\) and \(L\) be two simplicial complexes, and let \(v\in V(K)\) and \(w\in V(L)\) be vertices. The \emph{wedge} of \(K\) and \(L\), denoted by \(K\vee L\), is the simplicial complex obtained by identifying the vertices \(v\) and \(w\).

Let \(X\) and \(Y\) be two topological spaces. Two continuous maps \(f,g:X\to Y\) are said to be \emph{homotopic}, denoted by \(f\simeq g\), if there exists a continuous map \(H:X\times [0,1]\to Y\) such that \(H(x,0)=f(x)\) and \(H(x,1)=g(x)\) for every \(x\in X\). The map \(H\) is called a \emph{homotopy} from \(f\) to \(g\). The spaces \(X\) and \(Y\) are said to have the \emph{same homotopy type}, or equivalently, to be \emph{homotopy equivalent}, if there exist continuous maps \(f:X\to Y\) and \(g:Y\to X\) such that \(g\circ f\simeq \operatorname{id}_X\) and \(f\circ g\simeq \operatorname{id}_Y\). Let \(X\) and \(Y\) be two topological spaces. A continuous map \(f:X\to Y\) is said to be \emph{null-homotopic} if \(f\) is homotopic to a constant map, that is, if \(f\simeq c\) for some constant map \(c:X\to Y\). Let \(K\) and \(L\) be two simplicial complexes. If \(\lvert K\rvert\simeq \lvert L\rvert\) (homotopy equivalent), then we simply write \(K\simeq L\). If \(\lvert K\rvert\) has the homotopy type of a single point, then we simply write \(K\simeq *\). In this case, \(K\) is said to be \emph{contractible}. For further details and related concepts, we refer the reader to Hatcher~\cite{Hatcher}. 

\begin{proposition}[\cite{Hatcher}, Section 0.4]\label{cone and suspension}
Let $K$ be a subcomplex of a simplicial complex $X$. Suppose that $K$ is contractible in $X$, that is, the inclusion map, $i:A\hookrightarrow X$ is null-homotopic, then $X \cup \operatorname{cone}(K)\simeq X \vee \Sigma(K)$.
\end{proposition}

\begin{definition}\label{defi:IC}
    {\rm  The \emph{independence complex} of a simple graph $G$, denoted by $\operatorname{Ind}(G)$, is the simplicial complex whose faces are the independent sets of $G$.}
\end{definition}

\begin{definition}\label{defi:PM}
    {\rm The \emph{perfect matching complex} of a simple graph $G$, denoted by $\mathcal M_p(G)$, is the simplicial complex whose faces are the matchings of $G$ that can be extended to a perfect matching of $G$.}
\end{definition}

\begin{definition}\label{defi:bm}
{\rm Let $G$ be a  simple graph. A matching $M$ of $G$ is called a \emph{bad matching} if $M$ is a minimal matching that cannot be extended to a perfect matching of $G$.}

\end{definition}

Let $G_1 \sqcup G_2$ denote the disjoint union of two graphs, $G_1$ and $G_2$. Then
Ind$(G_1 \sqcup G_2) \simeq$  Ind$(G_1)\  \star$ Ind$(G_2).$ Observe that for a vertex $v\in V(G)$,  the link and the deletion of $v$ in Ind$(G)$ are the same as Ind$(G\setminus N[v])$ and Ind$(G\setminus v)$, respectively. Thus, we have that Ind$(G)$ is same as $(\{v\} \ \star$ Ind$(G\setminus N[v]))$ $\cup$ Ind$(G\setminus v)$. We now recall some known results that will be useful in proving our results.

\begin{proposition}[\cite{a12}]\label{ld}
Let $v$ be a vertex of a graph $G$. If the inclusion {\rm Ind$(G\setminus N[v])\hookrightarrow $ Ind$(G\setminus v)$} is null-homotopic, then we have {\rm $$ \operatorname{Ind}(G) \simeq \operatorname{Ind}(G\setminus v) \vee  \Sigma \operatorname{Ind}(G\setminus N[v]).$$}
\end{proposition}

\begin{proposition}[\cite{e08}]\label{easy}
Let $v$ and $w$ be a pair of distinct vertices of $G$ with $N(v)\subseteq N(w)$. Then {\rm Ind$(G) \simeq$ Ind$(G\setminus w)$.}   
\end{proposition}

The path graph $P_n$ is a graph with the vertex set $V(P_n)=\{v_1,v_2,\dots, v_n\}$ and the edge set $E(P_n)=\{v_iv_{i+1} \ |\ 1\le i\le n-1 \}$, where $n\ge 1$.

\begin{proposition}[\cite{k99}]\label{path}
   For $n\ge 1$,
   $$ \operatorname{ Ind}(P_n)\simeq \begin{cases}
       \mathbb{S}^{k-1} & \text{if } n=3k,3k-1,\\
       \ast & \text{if } n=3k+1.
   \end{cases}$$
\end{proposition}

The cycle graph $C_n$ is a graph with the vertex set $V(C_n)=\{v_1,v_2,\dots, v_n\}$ and the edge set $E(C_n)=\{v_iv_{i+1} \ |\ 1\le i\le n-1 \} \cup \{v_1v_n\}$, where $n\ge 3$.

\begin{proposition}[\cite{k99}]\label{cycle}
   For $n\ge 3$,
   $$ \operatorname{ Ind}(C_n)\simeq \begin{cases}
       \mathbb{S}^{k-1} \vee \mathbb{S}^{k-1} & \text{if } n=3k,\\
       \mathbb{S}^{k-1} & \text{if } n=3k\pm 1.
   \end{cases}$$
\end{proposition}

We use the following result from~\cite{b95} to study perfect matching complexes.

\begin{proposition}[\cite{b95}]\label{prop:bjorner}
Let $\Delta=\Delta_0\cup \Delta_1\cup \cdots \cup \Delta_n$ be a simplicial complex with subcomplexes $\Delta_i$, and assume that $\Delta_i\cap \Delta_j\subseteq \Delta_0$ for all $1\le i<j\le n$. 
\begin{enumerate}[(i)]
\item  If $\Delta_i$ is contractible for all $1\le i\le n$, then $$\Delta\simeq \Delta_0 \cup \bigcup_{1\le i\le n} \operatorname{cone} (\Delta_0\cap \Delta_i).$$

    \item If $\Delta_i$ is contractible for all $0\le i\le n$, then $$\Delta\simeq \vee_{1\le i\le n} \Sigma (\Delta_0\cap \Delta_i).$$

\end{enumerate}

\end{proposition}

\section{M\"obius and Circular Ladder Graphs}\label{defi}
 Let us first review the definition of the M\"obius ladder graph. It has an even number of vertices. Let us denote this graph by $M_{2n}$, where $2n$ is the number of vertices (Figure \ref{fig:ML}). The vertex set $V(M_{2n})= \{v_1,v_2,\dots, v_{2n}\}$ and the edge set $$E(M_{2n})=\{v_iv_{i+1},\  v_1v_{2n},\ v_jv_{j+n}\ |\ 1\le i\le 2n-1,\ 1\le j\le n\}.$$

\begin{figure}[ht]
\tikzstyle{vert}=[circle, draw, fill=black!100, inner sep=0pt, minimum width=3pt]
\tikzstyle{ver}=[]
\tikzstyle{extra}=[circle, draw, fill=black!50, inner sep=0pt, minimum width=4pt]
\tikzstyle{edge} = [draw,thick,-]
\centering
\begin{tikzpicture}[scale=1]

\foreach \x/\y/\z in {5/-1/0,5/1/1,3/-1/2,3/1/3,1/-1/4,1/1/5,-1/-1/6,-1/1/7,-3/-1/8,-3/1/9,-5/-1/10,-5/1/11}
{\node[vert] (a\z) at (\x,\y){};}

\foreach \x/\y in {2.5/-1,2/-1,1.5/-1,2.5/1,2/1,1.5/1}{\node[extra] () at (\x,\y){};}

\foreach \x/\y in {a0/a2,a6/a4,a8/a6,a1/a3,a7/a5,a9/a7,a0/a1,a2/a3,a4/a5,a6/a7,a8/a9,a0/a11,a1/a10,a10/a11,a8/a10,a9/a11}
{\draw [edge]  (\x) -- (\y);}

\foreach \x/\y/\z in {4/-1.3/c_{n-1},4/1.3/b_{n-1},0/-1.3/c_3,0/1.3/b_3,-2/-1.3/c_2,-2/1.3/b_2,-4/-1.3/c_1,-4/1.3/b_1,-5.3/0/a_1,-3.3/0/a_2,-1.3/0/a_3,1.3/0/a_4,3.5/0/a_{n-1},5.3/0/a_n, 1.8/0.6/d_2, 1.8/-0.6/d_1}
{\node[ver] () at (\x,\y){\small{$ \z$}};}

\node[ver] () at (5.25,-1.5){\small{$ v_{2n}$}};
\node[ver] () at (5.25,1.5){\small{$ v_n$}};
\node[ver] () at (3,-1.5){\small{$ v_{2n-1}$}};
\node[ver] () at (3,1.5){\small{$ v_{n-1}$}};
\node[ver] () at (1,-1.5){\small{$ v_{n+4}$}};
\node[ver] () at (1,1.5){\small{$ v_4$}};
\node[ver] () at (-1,-1.5){\small{$ v_{n+3}$}};
\node[ver] () at (-1,1.5){\small{$ v_3$}};
\node[ver] () at (-3,-1.5){\small{$ v_{n+2}$}};
\node[ver] () at (-3,1.5){\small{$ v_2$}};
\node[ver] () at (-5.25,1.5){\small{$ v_1$}};
\node[ver] () at (-5.25,-1.5){\small{$ v_{n+1}$}};


\end{tikzpicture}
\caption{M\"obius ladder graph $M_{2n}$ with $2n$ vertices.}\label{fig:ML}
\end{figure}
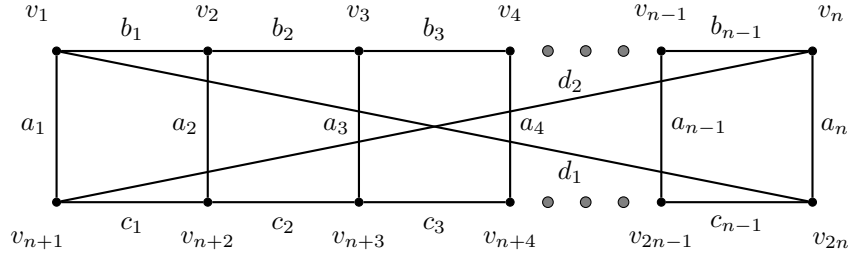

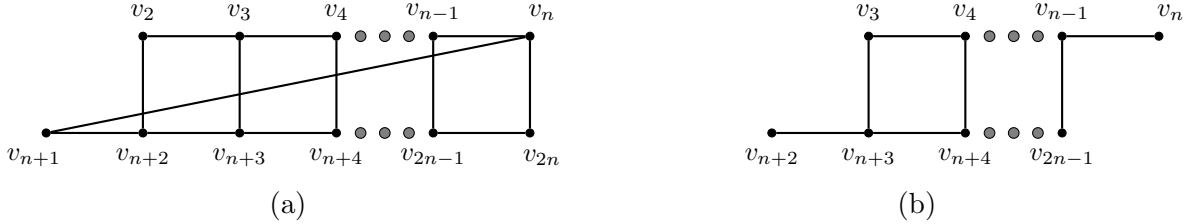
\begin{figure}[ht]
\tikzstyle{vert}=[circle, draw, fill=black!100, inner sep=0pt, minimum width=3pt]
\tikzstyle{ver}=[]
\tikzstyle{extra}=[circle, draw, fill=black!50, inner sep=0pt, minimum width=4pt]
\tikzstyle{edge} = [draw,thick,-]
\centering
\begin{tikzpicture}[scale=.64]

\begin{scope}[shift={(-10,0)}]
\foreach \x/\y/\z in {5/-1/0,5/1/1,3/-1/2,3/1/3,1/-1/4,1/1/5,-1/-1/6,-1/1/7,-3/-1/8,-3/1/9,-5/-1/10}
{\node[vert] (a\z) at (\x,\y){};}
\foreach \x/\y in {2.5/-1,2/-1,1.5/-1,2.5/1,2/1,1.5/1}{\node[extra] () at (\x,\y){};}

\foreach \x/\y in {a0/a2,a6/a4,a8/a6,a1/a3,a7/a5,a9/a7,a0/a1,a2/a3,a4/a5,a6/a7,a8/a9,a1/a10,a8/a10}
{\draw [edge]  (\x) -- (\y);}

\node[ver] () at (5.25,-1.5){\small{$ v_{2n}$}};
\node[ver] () at (5.25,1.5){\small{$ v_n$}};
\node[ver] () at (3,-1.5){\small{$ v_{2n-1}$}};
\node[ver] () at (3,1.5){\small{$ v_{n-1}$}};
\node[ver] () at (1,-1.5){\small{$ v_{n+4}$}};
\node[ver] () at (1,1.5){\small{$ v_4$}};
\node[ver] () at (-1,-1.5){\small{$ v_{n+3}$}};
\node[ver] () at (-1,1.5){\small{$ v_3$}};
\node[ver] () at (-3,-1.5){\small{$ v_{n+2}$}};
\node[ver] () at (-3,1.5){\small{$ v_2$}};

\node[ver] () at (-5.25,-1.5){\small{$ v_{n+1}$}};

\node[ver] () at (0,-2.5){(a)};
\end{scope}

\begin{scope}[shift={(3,0)}]
\foreach \x/\y/\z in {5/1/1,3/-1/2,3/1/3,1/-1/4,1/1/5,-1/-1/6,-1/1/7,-3/-1/8}
{\node[vert] (a\z) at (\x,\y){};}
\foreach \x/\y in {2.5/-1,2/-1,1.5/-1,2.5/1,2/1,1.5/1}{\node[extra] () at (\x,\y){};}

\foreach \x/\y in {a6/a4,a8/a6,a1/a3,a7/a5,a2/a3,a4/a5,a6/a7}
{\draw [edge]  (\x) -- (\y);}

\node[ver] () at (5.25,1.5){\small{$ v_n$}};
\node[ver] () at (3,-1.5){\small{$ v_{2n-1}$}};
\node[ver] () at (3,1.5){\small{$ v_{n-1}$}};
\node[ver] () at (1,-1.5){\small{$ v_{n+4}$}};
\node[ver] () at (1,1.5){\small{$ v_4$}};
\node[ver] () at (-1,-1.5){\small{$ v_{n+3}$}};
\node[ver] () at (-1,1.5){\small{$ v_3$}};
\node[ver] () at (-3,-1.5){\small{$ v_{n+2}$}};

\node[ver] () at (0,-2.5){(b)};
\end{scope}

\end{tikzpicture}
\caption{(a) $M_{2n}\setminus v_1$, (b) $M_{2n}\setminus N[v_1]$.}\label{fig:LD}
\end{figure}

Let us now review the definition of the circular ladder graph. It has an even number of vertices. Let us denote this graph by $\mathcal C_{2n}$, where $2n$ is the number of vertices (Figure \ref{fig:CL}). The vertex set $V(\mathcal C_{2n})= \{v_1 ,v_2 ,\dots, v_{2n} \}$ and the edge set $$E(\mathcal C_{2n})=\{v_i  v_{i+1} ,\  v_1  v_{n} ,\ v_{j}  v_{j+1} , \ v_{n+1}  v_{2n} ,\ v_k  v_{k+n} \ |\ 1\le i\le n-1,\ n+1\le j\le 2n-1, \ 1\le k\le n\}.$$ Note that $\mathcal C_{2n}$ is cartesian product of the $n$-cycle $C_n$ and the path $P_2$.

\begin{figure}[ht]
\tikzstyle{vert}=[circle, draw, fill=black!100, inner sep=0pt, minimum width=3pt]
\tikzstyle{ver}=[]
\tikzstyle{extra}=[circle, draw, fill=black!50, inner sep=0pt, minimum width=4pt]
\tikzstyle{edge} = [draw,thick,-]
\centering
\begin{tikzpicture}[scale=1]

\foreach \x/\y/\z in {5/-1/0,5/1/1,3/-1/2,3/1/3,1/-1/4,1/1/5,-1/-1/6,-1/1/7,-3/-1/8,-3/1/9,-5/-1/10,-5/1/11}
{\node[vert] (a\z) at (\x,\y){};}

\foreach \x/\y in {2.5/-1,2/-1,1.5/-1,2.5/1,2/1,1.5/1}{\node[extra] () at (\x,\y){};}

\foreach \x/\y in {a0/a2,a6/a4,a8/a6,a1/a3,a7/a5,a9/a7,a0/a1,a2/a3,a4/a5,a6/a7,a8/a9,a10/a11,a8/a10,a9/a11}
{\draw [edge]  (\x) -- (\y);}

\foreach \x/\y/\z in {4/-1.3/c_{n-1},4/1.3/b_{n-1},0/-1.3/c_3,0/1.3/b_3,-2/-1.3/c_2,-2/1.3/b_2,-4/-1.3/c_1,-4/1.3/b_1,-5.3/0/a_1,-3.3/0/a_2,-1.3/0/a_3,1.3/0/a_4,3.5/0/a_{n-1},5.3/0/a_n, 0/-2.3/c_n, 0/2.3/b_n}
{\node[ver] () at (\x,\y){\small{$ \z $}};}

\foreach \x/\y/\z/\w in {a11/a1/0/2,a10/a0/0/-2}
{\draw [edge] plot [smooth,tension=1.5] coordinates{(\x) (\z,\w) (\y)};}

\node[ver] () at (5.25,-1.5){\small{$ v_{2n}$}};
\node[ver] () at (5.25,1.5){\small{$ v_n$}};
\node[ver] () at (3,-1.5){\small{$ v_{2n-1}$}};
\node[ver] () at (3,1.5){\small{$ v_{n-1}$}};
\node[ver] () at (1,-1.5){\small{$ v_{n+4}$}};
\node[ver] () at (1,1.5){\small{$ v_4$}};
\node[ver] () at (-1,-1.5){\small{$ v_{n+3}$}};
\node[ver] () at (-1,1.5){\small{$ v_3$}};
\node[ver] () at (-3,-1.5){\small{$ v_{n+2}$}};
\node[ver] () at (-3,1.5){\small{$ v_2$}};
\node[ver] () at (-5.25,1.5){\small{$ v_1$}};
\node[ver] () at (-5.25,-1.5){\small{$ v_{n+1}$}};


\end{tikzpicture}
\caption{Circular ladder graph $\mathcal C_{2n}$ with $2n$ vertices.}\label{fig:CL}
\end{figure}
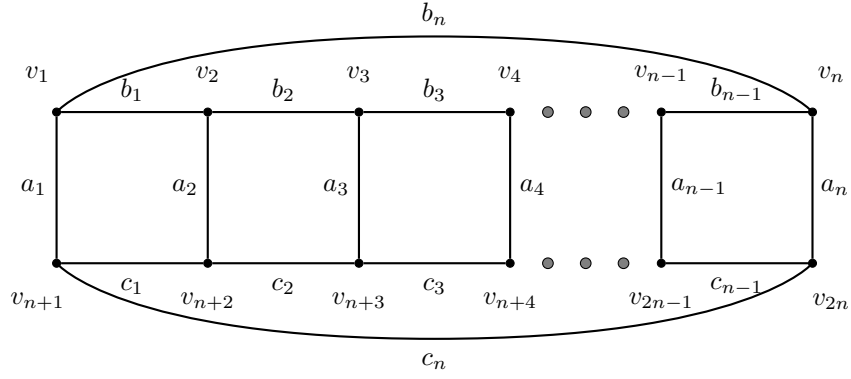

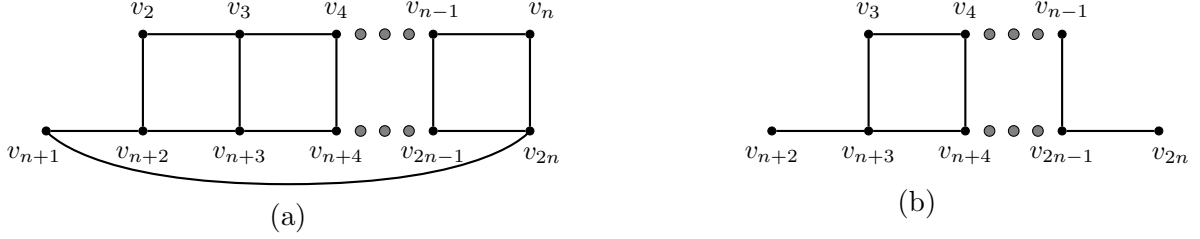
\begin{figure}[ht]
\tikzstyle{vert}=[circle, draw, fill=black!100, inner sep=0pt, minimum width=3pt]
\tikzstyle{ver}=[]
\tikzstyle{extra}=[circle, draw, fill=black!50, inner sep=0pt, minimum width=4pt]
\tikzstyle{edge} = [draw,thick,-]
\centering
\begin{tikzpicture}[scale=.64]

\begin{scope}[shift={(-10,0)}]
\foreach \x/\y/\z in {5/-1/0,5/1/1,3/-1/2,3/1/3,1/-1/4,1/1/5,-1/-1/6,-1/1/7,-3/-1/8,-3/1/9,-5/-1/10}
{\node[vert] (a\z) at (\x,\y){};}
\foreach \x/\y in {2.5/-1,2/-1,1.5/-1,2.5/1,2/1,1.5/1}{\node[extra] () at (\x,\y){};}

\foreach \x/\y in {a0/a2,a6/a4,a8/a6,a1/a3,a7/a5,a9/a7,a0/a1,a2/a3,a4/a5,a6/a7,a8/a9,a8/a10}
{\draw [edge]  (\x) -- (\y);}

\node[ver] () at (5.25,-1.5){\small{$ v_{2n}$}};
\node[ver] () at (5.25,1.5){\small{$ v_n$}};
\node[ver] () at (3,-1.5){\small{$ v_{2n-1}$}};
\node[ver] () at (3,1.5){\small{$ v_{n-1}$}};
\node[ver] () at (1,-1.5){\small{$ v_{n+4}$}};
\node[ver] () at (1,1.5){\small{$ v_4$}};
\node[ver] () at (-1,-1.5){\small{$ v_{n+3}$}};
\node[ver] () at (-1,1.5){\small{$ v_3$}};
\node[ver] () at (-3,-1.5){\small{$ v_{n+2}$}};
\node[ver] () at (-3,1.5){\small{$ v_2$}};

\node[ver] () at (-5.25,-1.5){\small{$ v_{n+1}$}};

\foreach \x/\y/\z/\w in {a10/a0/0/-2.1}
{\draw [edge] plot [smooth,tension=1.5] coordinates{(\x) (\z,\w) (\y)};}

\node[ver] () at (0,-2.8){(a)};
\end{scope}

\begin{scope}[shift={(3,0)}]
\foreach \x/\y/\z in {5/-1/0,3/-1/2,3/1/3,1/-1/4,1/1/5,-1/-1/6,-1/1/7,-3/-1/8}
{\node[vert] (a\z) at (\x,\y){};}
\foreach \x/\y in {2.5/-1,2/-1,1.5/-1,2.5/1,2/1,1.5/1}{\node[extra] () at (\x,\y){};}

\foreach \x/\y in {a6/a4,a8/a6,a0/a2,a7/a5,a2/a3,a4/a5,a6/a7}
{\draw [edge]  (\x) -- (\y);}

\node[ver] () at (5.25,-1.5){\small{$ v_{2n}$}};
\node[ver] () at (3,-1.5){\small{$ v_{2n-1}$}};
\node[ver] () at (3,1.5){\small{$ v_{n-1}$}};
\node[ver] () at (1,-1.5){\small{$ v_{n+4}$}};
\node[ver] () at (1,1.5){\small{$ v_4$}};
\node[ver] () at (-1,-1.5){\small{$ v_{n+3}$}};
\node[ver] () at (-1,1.5){\small{$ v_3$}};
\node[ver] () at (-3,-1.5){\small{$ v_{n+2}$}};

\node[ver] () at (0,-2.5){(b)};
\end{scope}

\end{tikzpicture}
\caption{(a) $\mathcal C_{2n}\setminus v_1$, (b) $\mathcal C_{2n}\setminus N[v_1]$.}\label{fig:CLD}
\end{figure}

\section{Independence Complex}\label{Section:3}

In this section, we study homotopy types of the independence complexes of the M\"obius and circular ladder graphs. Let us first study the homotopy type of the independence complex of the M\"obius ladder graph.

\subsection{Independence complex of M\"obius ladder graph}

\begin{lemma}\label{deletion}
For $n\ge 2$,
$$ \operatorname{ Ind}(M_{2n}\setminus v_1)\simeq \begin{cases}
       \mathbb{S}^{2k-1} & \text{if } n=4k, \\ 
       \ast & \text{if } n=4k+1, \\ 
       \mathbb{S}^{2k} \vee \mathbb{S}^{2k} & \text{if } n=4k+2, \\ 
        
       \mathbb{S}^{2k} & \text{if } n=4k+3.
   \end{cases}$$
    
\end{lemma}

\begin{proof}
In $M_{2n}\setminus v_1$, consider $v_2$, and observe that $N(v_2)\subset N(v_{n+3})$. Thus, Ind$(M_{2n}\setminus v_1)\simeq$ Ind$(M_{2n}\setminus \{v_1,v_{n+3}\})$ due to Proposition \ref{easy}. In $M_{2n}\setminus \{v_1,v_{n+3}\}$, $N(v_{n+4})\subset N(v_5)$. Applying Proposition \ref{easy} on $M_{2n}\setminus \{v_1,v_{n+3}\}$, we get that Ind$(M_{2n}\setminus v_1)\simeq$ Ind$(M_{2n}\setminus \{v_1,v_{n+3}\})\simeq$ Ind$(M_{2n}\setminus \{v_1,v_{n+3},v_5\})$. Therefore, applying Proposition \ref{easy} iteratively, we finally get that $$\operatorname{Ind}(M_{2n}\setminus v_1)\simeq \begin{cases}
       \operatorname{Ind}(M_{2n}\setminus \{v_1,v_{n+3},v_5,v_{n+7},\dots, v_{n-3}, v_{2n-1}\}) & \text{if } n=4k, \\ 
       \operatorname{Ind}(M_{2n}\setminus \{v_1,v_{n+3},v_5,v_{n+7},\dots, v_{n-4}, v_{2n-2}, v_n\}) & \text{if } n=4k+1, \\ 
       \operatorname{Ind}(M_{2n}\setminus \{v_1,v_{n+3},v_5,v_{n+7},\dots, v_{n-5}, v_{2n-3}, v_{n-1}\}) & \text{if } n=4k+2, \\ 
        
       \operatorname{Ind}(M_{2n}\setminus \{v_1,v_{n+3},v_5,v_{n+7},\dots, v_{n-6}, v_{2n-4}, v_{n-2},v_{2n}\}) & \text{if } n=4k+3.
   \end{cases}$$

For $n=4k$, note that $N(v_{2n})=\{v_n\}\subset N(v_{n-1})$ in $M_{2n}\setminus \{v_1,v_{n+3},v_5,v_{n+7},\dots, v_{n-3}, v_{2n-1}\}$. Therefore, $\operatorname{Ind}(M_{2n}\setminus v_1)=
       \operatorname{Ind}(M_{2n}\setminus \{v_1,v_{n+3},v_5,v_{n+7},\dots, v_{n-3}, v_{2n-1},v_{n-1}\})$ for $n=4k$. Observe that $M_{2n}\setminus \{v_1,v_{n+3},v_5,v_{n+7},\dots, v_{n-3}, v_{2n-1},v_{n-1}\}$ is a path with endpoints $v_{n-2}$ and $v_{2n}$. Since $M_{2n}$ has $2n=8k$ vertices and the cardinality of the set $ \{v_1,v_{n+3},v_5,v_{n+7},\dots, v_{n-3}, v_{2n-1},\\v_{n-1}\}$ is $2k+1$, $M_{2n}\setminus \{v_1,v_{n+3},v_5,v_{n+7},\dots, v_{n-3}, v_{2n-1},v_{n-1}\}$ is the path $P_{6k-1}$.

       For $n=4k+1$, $M_{2n}\setminus \{v_1,v_{n+3},v_5,v_{n+7},\dots, v_{n-4}, v_{2n-2}, v_n\}$ is a path with endpoints $v_{n+1}$ and $v_{2n}$. Since the cardinality of the set $\{v_1,v_{n+3},v_5,v_{n+7},\dots, v_{n-4}, v_{2n-2}, v_n\}$ is $2k+1$, $M_{2n}\setminus \{v_1,v_{n+3},v_5,v_{n+7},\dots, v_{n-4}, v_{2n-2}, v_n\}$ is the path $P_{6k+1}$.

       For $n=4k+2$, $M_{2n}\setminus \{v_1,v_{n+3},v_5,v_{n+7},\dots, v_{n-5}, v_{2n-3}, v_{n-1}\}$ is a cycle with $2(4k+2)-2k-1=6k+3$ vertices. Therefore, $M_{2n}\setminus \{v_1,v_{n+3},v_5,v_{n+7},\dots, v_{n-5}, v_{2n-3}, v_{n-1}\}$ is the cycle $C_{6k+3}$.

       For $n=4k+3$, $M_{2n}\setminus \{v_1,v_{n+3},v_5,v_{n+7},\dots, v_{n-6}, v_{2n-4}, v_{n-2},v_{2n}\}$ is a cycle with $2(4k+3)-2k-2=6k+4$ vertices. Therefore, $M_{2n}\setminus \{v_1,v_{n+3},v_5,v_{n+7},\dots, v_{n-6}, v_{2n-4}, v_{n-2},v_{2n}\}$ is the cycle $C_{6k+4}$. In summary, we get the following,

    $$\operatorname{Ind}(M_{2n}\setminus v_1)\simeq \begin{cases}
       \operatorname{Ind}(P_{6k-1}) & \text{if } n=4k, \\ 
       \operatorname{Ind}(P_{6k+1}) & \text{if } n=4k+1, \\ 
       \operatorname{Ind}(C_{3(2k+1)}) & \text{if } n=4k+2, \\ 
        
       \operatorname{Ind}(C_{3(2k+1)+1}) & \text{if } n=4k+3.
   \end{cases}$$   
Now, the result follows from Propositions \ref{path} and \ref{cycle}.
\end{proof}

\begin{lemma}\label{link}
For $n\ge 2$,
$$ \operatorname{ Ind}(M_{2n}\setminus N[v_1])\simeq \begin{cases}
       \ast & \text{if } n=4k,4k+3, \\ 
        
       \mathbb{S}^{2k-1} & \text{if } n=4k+1,4k+2.
   \end{cases}$$
    
\end{lemma}

\begin{proof}
In $(M_{2n}\setminus N[v_1])$, consider $v_3$ and observe that $N(v_3)\subset N(v_{n+4})$. Thus, Ind$(M_{2n}\setminus N[v_1])\simeq$ Ind$(M_{2n}\setminus (N[v_1]\cup\{v_{n+4}\}))$ due to Proposition \ref{easy}. In $M_{2n}\setminus (N[v_1]\cup\{v_{n+4}\})$, $N(v_{n+5})\subset N(v_6)$. Applying Proposition \ref{easy} on $M_{2n}\setminus (N[v_1]\cup\{v_{n+4}\})$, we get that Ind$(M_{2n}\setminus N[v_1])\simeq$ Ind$(M_{2n}\setminus (N[v_1]\cup\{v_{n+4}\}))\simeq$ Ind$(M_{2n}\setminus (N[v_1]\cup\{v_{n+4}, v_6\}))$. Therefore, applying Proposition \ref{easy} iteratively, we finally get that $$\operatorname{Ind}(M_{2n}\setminus N[v_1])\simeq \begin{cases}
       \operatorname{Ind}(M_{2n}\setminus (N[v_1]\cup \{v_{n+4},v_6,v_{n+8},\dots, v_{2n-4}, v_{n-2}\})) & \text{if } n=4k, \\ 
       \operatorname{Ind}(M_{2n}\setminus (N[v_1]\cup \{v_{n+4},v_6,v_{n+8},\dots, v_{n-3}, v_{2n-1}\})) & \text{if } n=4k+1, \\ 
       \operatorname{Ind}(M_{2n}\setminus (N[v_1]\cup \{v_{n+4},v_6,v_{n+8},\dots, v_{2n-2}, v_{n}\})) & \text{if } n=4k+2, \\ 
        
       \operatorname{Ind}(M_{2n}\setminus (N[v_1]\cup \{v_{n+4},v_6,v_{n+8},\dots, v_{2n-3}, v_{n-1}\})) & \text{if } n=4k+3.
   \end{cases}$$

 For $n=4k$, $M_{2n}\setminus (N[v_1]\cup \{v_{n+4},v_6,v_{n+8},\dots, v_{2n-4}, v_{n-2}\})$ is a path with endpoints $v_{n+2}$ and $v_{n}$. Since the cardinality of the set $N[v_1]\cup \{v_{n+4},v_6,v_{n+8},\dots, v_{2n-4}, v_{n-2}\}$ is $4+2(k-1)=2k+2$, $M_{2n}\setminus (N[v_1]\cup \{v_{n+4},v_6,v_{n+8},\dots, v_{2n-4}, v_{n-2}\})$ is the path $P_{6k-2}=P_{3(2k-1)+1}$.

 For $n=4k+1$, $M_{2n}\setminus (N[v_1]\cup \{v_{n+4},v_6,v_{n+8},\dots, v_{n-3}, v_{2n-1}\})$ is a path with endpoints $v_{n+2}$ and $v_{n}$. Since the cardinality of the set $N[v_1]\cup \{v_{n+4},v_6,v_{n+8},\dots, v_{n-3}, v_{2n-1}\}$ is $4+2(k-1)+1=2k+3$, $M_{2n}\setminus (N[v_1]\cup \{v_{n+4},v_6,v_{n+8},\dots, v_{n-3}, v_{2n-1}\})$ is the path $P_{6k-1}$.

  For $n=4k+2$, $M_{2n}\setminus (N[v_1]\cup \{v_{n+4},v_6,v_{n+8},\dots, v_{2n-2}, v_n\})$ is a path with endpoints $v_{n+2}$ and $v_{2n-1}$. Since the cardinality of the set $N[v_1]\cup \{v_{n+4},v_6,v_{n+8},\dots,  v_{2n-2}, v_n\}$ is $4+2(k-1)+2=2k+4$, $M_{2n}\setminus (N[v_1]\cup \{v_{n+4},v_6,v_{n+8},\dots, v_{2n-2}, v_n\})$ is the path $P_{6k}$.

For $n=4k+3$, note that $v_n$ is an isolated vertex in $M_{2n}\setminus (N[v_1]\cup \{v_{n+4},v_6,v_{n+8},\dots, v_{2n-3},\\ v_{n-1}\})$. Therefore, Ind$(M_{2n}\setminus N[v_1])$ is contractible.  In summary, we get the following,

    $$\operatorname{Ind}(M_{2n}\setminus N[v_1])\simeq \begin{cases}
       \operatorname{Ind}(P_{3(2k-1)+1}) & \text{if } n=4k, \\ 
       \operatorname{Ind}(P_{6k-1}) & \text{if } n=4k+1, \\ 
       \operatorname{Ind}(P_{6k}) & \text{if } n=4k+2, \\ 
        
       \ast & \text{if } n=4k+3.
   \end{cases}$$   
Now, the result follows from Proposition \ref{path}.
\end{proof}

\begin{theorem}\label{IC}
For $n\ge 2$,
$$ \operatorname{ Ind}(M_{2n})\simeq \begin{cases}
       \mathbb{S}^{2k-1} & \text{if } n=4k, \\ \bigvee_3  \mathbb{S}^{2k} & \text{if } n=4k+2 \\ \mathbb{S}^{2k} & \text{if } n=4k+1,4k+3.
   \end{cases}$$
    
\end{theorem}
\begin{proof}
    Consider $M_{2n}\setminus v_1$ and $M_{2n}\setminus N[v_1]$. Due to Lemma \ref{link}, {\rm Ind}$(M_{2n}\setminus N[v_1])\simeq \ast$ for $n=4k$. Thus, Proposition \ref{ld} can be applied on $M_{2n}$ w.r.t. $v_1$ for $n=4k$. Now the result follows from Lemma \ref{deletion} for $n=4k$.

    For $n=4k+2$, Ind$(M_{2n}\setminus N[v_1])\simeq \mathbb{S}^{2k-1}$ by Lemma \ref{link} and Ind$(M_{2n}\setminus v_1)\simeq \mathbb{S}^{2k}\vee \mathbb{S}^{2k}$ by Lemma \ref{deletion}. Thus, Proposition \ref{ld} can be applied on $M_{2n}$ w.r.t. $v_1$ for $n=4k+2$. Hence, the result for $n$ even follows. The result for $n$ odd also follows from Lemma \ref{deletion}, \ref{link} and Proposition \ref{ld}.
    \end{proof}

\subsection{Independence complex of circular ladder graph}
Let us now study the homotopy type of the independence complex of the circular ladder graph.

\begin{lemma}\label{deletionCL}
For $n\ge 2$,
$$ \operatorname{ Ind}(\mathcal C_{2n}\setminus v_1)\simeq \begin{cases}
       \mathbb{S}^{2k-1}\vee \mathbb{S}^{2k-1} & \text{if } n=4k, \\ 
      \mathbb{S}^{2k-1} & \text{if } n=4k+1, \\ 
       \mathbb{S}^{2k}  & \text{if } n=4k+2, \\ 
        
       \ast & \text{if } n=4k+3.
   \end{cases}$$
   \end{lemma}
   
\begin{proof}
   In $\mathcal C_{2n}\setminus v_1$, consider $v_2$, and observe that $N(v_2)\subset N(v_{n+3})$. Thus, Ind$(\mathcal C_{2n}\setminus v_1)\simeq$ Ind$(\mathcal C_{2n}\setminus \{v_1,v_{n+3}\})$ due to Proposition \ref{easy}. In $\mathcal C_{2n}\setminus \{v_1,v_{n+3}\}$, $N(v_{n+4})\subset N(v_5)$. Applying Proposition \ref{easy} on $\mathcal C_{2n}\setminus \{v_1,v_{n+3}\}$, we get that Ind$(\mathcal C_{2n}\setminus v_1)\simeq$ Ind$(\mathcal C_{2n}\setminus \{v_1,v_{n+3}\})\simeq$ Ind$(\mathcal C_{2n}\setminus \{v_1,v_{n+3},v_5\})$. Therefore, applying Proposition \ref{easy} iteratively, we finally get that $$\operatorname{Ind}(\mathcal C_{2n}\setminus v_1)\simeq \begin{cases}
       \operatorname{Ind}(\mathcal C_{2n}\setminus \{v_1,v_{n+3},v_5,v_{n+7},\dots, v_{n-3}, v_{2n-1}\}) & \text{if } n=4k, \\ 
       \operatorname{Ind}(\mathcal C_{2n}\setminus \{v_1,v_{n+3},v_5,v_{n+7},\dots, v_{n-4}, v_{2n-2}, v_n\}) & \text{if } n=4k+1, \\ 
       \operatorname{Ind}(\mathcal C_{2n}\setminus \{v_1,v_{n+3},v_5,v_{n+7},\dots, v_{n-5}, v_{2n-3}, v_{n-1}\}) & \text{if } n=4k+2, \\ 
        
       \operatorname{Ind}(\mathcal C_{2n}\setminus \{v_1,v_{n+3},v_5,v_{n+7},\dots, v_{n-6}, v_{2n-4}, v_{n-2},v_{2n}\}) & \text{if } n=4k+3.
   \end{cases}$$

   For $n=4k$, $\mathcal C_{2n}\setminus \{v_1,v_{n+3},v_5,v_{n+7},\dots, v_{n-3}, v_{2n-1}\}$ is a cycle with $2(4k)-2k=6k$ vertices. Therefore, $\mathcal C_{2n}\setminus \{v_1,v_{n+3},v_5,v_{n+7},\dots, v_{n-3}, v_{2n-1}\}$ is the cycle $C_{6k}$.

    For $n=4k+1$, $\mathcal C_{2n}\setminus \{v_1,v_{n+3},v_5,v_{n+7},\dots, v_{n-4}, v_{2n-2}, v_n\}$ is a cycle with $2(4k+1)-2k-1=6k+1$ vertices. Therefore, $\mathcal C_{2n}\setminus \{v_1,v_{n+3},v_5,v_{n+7},\dots, v_{n-4}, v_{2n-2}, v_n\}$ is the cycle $C_{6k+1}$.

For $n=4k+2$, note that $N(v_{n})=\{v_{2n}\}\subset N(v_{n+1})$ in $\mathcal C_{2n}\setminus \{v_1,v_{n+3},v_5,v_{n+7},\dots, v_{n-5},\\ v_{2n-3}, v_{n-1}\}$. Therefore, $\operatorname{Ind}(\mathcal C_{2n}\setminus v_1)=
       \operatorname{Ind}(\mathcal C_{2n}\setminus \{v_1,v_{n+3},v_5,v_{n+7},\dots, v_{n-5}, v_{2n-3},v_{n-1}, v_{n+1}\})$ for $n=4k+2$. Observe that $\mathcal C_{2n}\setminus \{v_1,v_{n+3},v_5,v_{n+7},\dots, v_{n-5}, v_{2n-3},v_{n-1}, v_{n+1}\}$ is a path with endpoints $v_{n+2}$ and $v_{n}$. Since $\mathcal C_{2n}$ has $2n=8k+4$ vertices and the cardinality of the set $ \{v_1,v_{n+3},v_5,v_{n+7},\dots, v_{n-5}, v_{2n-3},v_{n-1}, v_{n+1}\}$ is $2k+2$, $\mathcal C_{2n}\setminus \{v_1,v_{n+3},v_5,v_{n+7},\dots, v_{n-5}, v_{2n-3},\\v_{n-1}, v_{n+1}\}$ is the path $P_{6k+2}$.

       For $n=4k+3$, $\mathcal C_{2n}\setminus \{v_1,v_{n+3},v_5,v_{n+7},\dots, v_{n-6}, v_{2n-4}, v_{n-2},v_{2n}\}$ is a path with endpoints $v_{n+1}$ and $v_{n}$. Since the cardinality of the set $\{v_1,v_{n+3},v_5,v_{n+7},\dots, v_{n-6}, v_{2n-4}, v_{n-2},v_{2n}\}$ is $2k+2$, $\mathcal C_{2n}\setminus \{v_1,v_{n+3},v_5,v_{n+7},\dots, v_{n-6}, v_{2n-4}, v_{n-2},v_{2n}\}$ is the path $P_{6k+4}$. In summary, we get the following,

    $$\operatorname{Ind}(\mathcal C_{2n}\setminus v_1)\simeq \begin{cases}
       \operatorname{Ind}(C_{6k}) & \text{if } n=4k, \\ 
       \operatorname{Ind}(C_{6k+1}) & \text{if } n=4k+1, \\ 
       \operatorname{Ind}(P_{(3(2k+1)-1)}) & \text{if } n=4k+2, \\ 
        
       \operatorname{Ind}(P_{3(2k+1)+1}) & \text{if } n=4k+3.
   \end{cases}$$   
Now, the result follows from Propositions \ref{path} and \ref{cycle}.
\end{proof}

   \begin{lemma}\label{linkCL}
For $n\ge 2$,
$$ \operatorname{ Ind}(\mathcal C_{2n}\setminus N[v_1])\simeq \begin{cases}
       \mathbb{S}^{2k-2} & \text{if } n=4k, \\ 
        
        \ast & \text{if } n=4k+1,4k+2, \\
        
        \mathbb{S}^{2k} & \text{if } n=4k+3.
   \end{cases}$$
    
\end{lemma}

\begin{proof}
In $(\mathcal C_{2n}\setminus N[v_1])$, consider $v_3$ and observe that $N(v_3)\subset N(v_{n+4})$. Thus, Ind$(\mathcal C_{2n}\setminus N[v_1])\simeq$ Ind$(\mathcal C_{2n}\setminus (N[v_1]\cup\{v_{n+4}\}))$ due to Proposition \ref{easy}. In $\mathcal C_{2n}\setminus (N[v_1]\cup\{v_{n+4}\})$, $N(v_{n+5})\subset N(v_6)$. Applying Proposition \ref{easy} on $\mathcal C_{2n}\setminus (N[v_1]\cup\{v_{n+4}\})$, we get that Ind$(\mathcal C_{2n}\setminus N[v_1])\simeq$ Ind$(\mathcal C_{2n}\setminus (N[v_1]\cup\{v_{n+4}\}))\simeq$ Ind$(\mathcal C_{2n}\setminus (N[v_1]\cup\{v_{n+4}, v_6\}))$. Therefore, applying Proposition \ref{easy} iteratively, we finally get that $$\operatorname{Ind}(\mathcal C_{2n}\setminus N[v_1])\simeq \begin{cases}
       \operatorname{Ind}(\mathcal C_{2n}\setminus (N[v_1]\cup \{v_{n+4},v_6,v_{n+8},\dots, v_{n-2},v_{2n}\})) & \text{if } n=4k, \\ 
       \operatorname{Ind}(\mathcal C_{2n}\setminus (N[v_1]\cup \{v_{n+4},v_6,v_{n+8},\dots, v_{n-3}, v_{2n-1}\})) & \text{if } n=4k+1, \\ 
       \operatorname{Ind}(\mathcal C_{2n}\setminus (N[v_1]\cup \{v_{n+4},v_6,v_{n+8},\dots, v_{n-4}, v_{2n-2}\})) & \text{if } n=4k+2, \\ 
        
       \operatorname{Ind}(\mathcal C_{2n}\setminus (N[v_1]\cup \{v_{n+4},v_6,v_{n+8},\dots, v_{2n-3}, v_{n-1}\})) & \text{if } n=4k+3.
   \end{cases}$$

 For $n=4k$, $\mathcal C_{2n}\setminus (N[v_1]\cup \{v_{n+4},v_6,v_{n+8},\dots, v_{n-2},v_{2n}\})$ is a path with endpoints $v_{n+2}$ and $v_{n-1}$. Since the cardinality of the set $N[v_1]\cup \{v_{n+4},v_6,v_{n+8},\dots, v_{n-2},v_{2n}\}$ is $4+2(k-1)+1=2k+3$, $\mathcal C_{2n}\setminus (N[v_1]\cup \{v_{n+4},v_6,v_{n+8},\dots, v_{n-2},v_{2n}\})$ is the path $P_{6k-3}$.

 For $n=4k+1$, note that $v_{2n}$ is an isolated vertex in $\mathcal C_{2n}\setminus (N[v_1]\cup \{v_{n+4},v_6,v_{n+8},\dots, v_{n-3}, \\v_{2n-1}\})$. Therefore, Ind$(\mathcal C_{2n}\setminus N[v_1])$ is contractible.

  For $n=4k+2$, $\mathcal C_{2n}\setminus (N[v_1]\cup \{v_{n+4},v_6,v_{n+8},\dots, v_{n-4}, v_{2n-2}\})$ is a path with endpoints $v_{n+2}$ and $v_{2n}$. Since the cardinality of the set $N[v_1]\cup \{v_{n+4},v_6,v_{n+8},\dots, v_{n-4}, v_{2n-2}\}$ is $4+2(k-1)+1=2k+3$, $\mathcal C_{2n}\setminus (N[v_1]\cup \{v_{n+4},v_6,v_{n+8},\dots, v_{n-4}, v_{2n-2}\})$ is the path $P_{6k+1}$.

For $n=4k+3$, $\mathcal C_{2n}\setminus (N[v_1]\cup \{v_{n+4},v_6,v_{n+8},\dots, v_{2n-3}, v_{n-1}\})$ is a path with endpoints $v_{n+2}$ and $v_{2n}$. Since the cardinality of the set $N[v_1]\cup \{v_{n+4},v_6,v_{n+8},\dots, v_{2n-3}, v_{n-1}\}$ is $4+2k$, $\mathcal C_{2n}\setminus (N[v_1]\cup \{v_{n+4},v_6,v_{n+8},\dots, v_{2n-3}, v_{n-1}\})$ is the path $P_{6k+2}$. In summary, we get the following,

    $$\operatorname{Ind}(\mathcal C_{2n}\setminus N[v_1])\simeq \begin{cases}
       \operatorname{Ind}(P_{3(2k-1)}) & \text{if } n=4k, \\ 
        \ast & \text{if } n=4k+1, \\ 
       \operatorname{Ind}(P_{6k+1}) & \text{if } n=4k+2, \\ 
        
       \operatorname{Ind}(P_{(3(2k+1)-1)}) & \text{if } n=4k+3.
   \end{cases}$$   
Now, the result follows from Proposition \ref{path}.
\end{proof}

   \begin{theorem}\label{ICCL}
For $n\ge 2$,
$$ \operatorname{ Ind}(\mathcal C_{2n})\simeq \begin{cases}
    \bigvee_3 \mathbb{S}^{2k-1} & \text{if } n=4k, \\   \mathbb{S}^{2k-1} & \text{if } n=4k+1 \\ \mathbb{S}^{2k} & \text{if } n=4k+2,\\
    \mathbb{S}^{2k+1} & \text{if } n=4k+3.
   \end{cases}$$
    
\end{theorem}
 \begin{proof}   
Consider $\mathcal C_{2n}\setminus v_1$ and $\mathcal C_{2n}\setminus N[v_1]$. Note that {\rm Ind}$(\mathcal C_{2n}\setminus N[v_1])$ is null-homotopic in {\rm Ind}$(\mathcal C_{2n}\setminus v_1)$ due to Lemma \ref{deletionCL} and \ref{linkCL}. Thus, Proposition \ref{ld} can be applied on $\mathcal C_{2n}$ w.r.t. $v_1$. Hence, it is proved.
\end{proof}

\section{Perfect Matching Complex}\label{Section:4}

In this section, we study the homotopy type of the perfect matching complexes of the M\"obius and circular ladder graphs. Let us start with the M\"obius ladder graph.

\subsection{Perfect matching complex of M\"obius ladder graph}

\begin{lemma}\label{lemma:extend}
Let $S$ be the set $$\{\{b_i,c_{i+1}\},\ \{c_i,b_{i+1}\},\ \{d_1,c_1\}, \ \{d_1,b_{n-1}\},\ \{d_2,b_1\},\ \{d_2,c_{n-1}\}\ |\ 1\le i\le n-2 \}.$$ If $\sigma$ is a matching of $M_{2n}$ such that $s \nsubseteq \sigma$  for all $s\in S$, then $\sigma$ can be extended to a perfect matching of $M_{2n}$ for all $n\ge 2$.
\end{lemma}
\begin{proof}
  Let $\sigma$ is a matching of $M_{2n}$ such that $s \nsubseteq \sigma$ for all $s\in S$. Let $V(\sigma)\subset V(M_{2n})$ be the set of all vertices of $M_{2n}$ that are endpoints of the edges of $\sigma$. Consider the edge $a_j\in E(M_{2n})$ for $1\le j\le n$. For notation, ref Figure \ref{fig:ML}. The endpoints $v_j$ and $v_{n+j}$ of $a_j$ satisfy exactly one of the following four conditions:
\begin{enumerate}[$(i)$]
    \item $v_j, v_{n+j} \in V(\sigma)$;
    \item $v_j \in V(\sigma)$ and $v_{n+j} \notin V(\sigma)$;
    \item $v_j \notin V(\sigma)$ and $v_{n+j} \in V(\sigma)$;
    \item $v_j, v_{n+j} \notin V(\sigma)$.
\end{enumerate}

We want to extend the matching $\sigma$ to a perfect matching. Therefore, when condition $(i)$ arises, look for some other $a_j$ that satisfies conditions $(ii)$, $(iii)$ or $(iv)$. If $a_j$ satisfies condition $(iv)$, then add $a_j$ to the matching $\sigma$. Note that conditions $(ii)$ and $(iii)$ are similar, so we consider condition $(ii)$. 

Let $a_j$ satisfy the condition $(ii)$, for $3\le j\le n-2$. This implies that $a_j\notin \sigma$. Since $a_j\notin \sigma$ and $v_j\in V(\sigma)$, either $b_{j-1}\in \sigma$ or $b_{j}\in \sigma$. Let us first assume that $b_{j}\in \sigma$. This implies that $a_{j+1}\notin \sigma$. Since $v_{n+j}\notin V(\sigma)$, $c_j\notin \sigma$. Since $b_{j}\in \sigma$ and $\{b_{j},c_{j+1}\} \nsubseteq \sigma $, $c_{j+1}\notin \sigma$. Now, since $c_j,c_{j+1},a_{j+1} \notin \sigma$, we have that $v_{n+j+1}\notin V(\sigma)$. Therefore, we can add $c_j$ to the matching $\sigma$. Let us now assume that $b_{j-1}\in \sigma$. Since $b_{j-1}\in \sigma$ and $\{c_{j-2},b_{j-1}\} \nsubseteq \sigma  $, $a_{j-1},c_{j-2} \notin \sigma$. Since $v_{n+j}\notin V(\sigma)$, $c_{j-1}\notin \sigma$. Now, since $c_{j-1},c_{j-2},a_{j-1} \notin \sigma$, we have that $v_{n+j-1}\notin V(\sigma)$. Therefore, we can add $c_{j-1}$ to the matching $\sigma$.

For $j\in \{1,2,n-1,n\}$, the argument follows similarly to the above. Let us consider the case when $j=1$. Since $a_1$ satisfies condition (ii), $a_1\notin \sigma$. Since $v_1\in V(\sigma)$, either $b_{1}\in \sigma$ or $d_{1}\in \sigma$. Let us first assume that $b_{1}\in \sigma$. This implies that $a_{2}\notin \sigma$. Since $b_{1}\in \sigma$ and $\{b_1,c_2\} \nsubseteq \sigma$, $c_{2}\notin \sigma$. Since $v_{n+1}\notin V(\sigma)$, $c_1\notin \sigma$.  Now, since $c_1,c_{2},a_{2} \notin \sigma$, we have that $v_{n+2}\notin V(\sigma)$. Therefore, we can add $c_1$ to the matching $\sigma$. Let us now assume that $d_{1}\in \sigma$. Since $d_{1}\in \sigma$ and $\{d_1,b_{n-1}\} \nsubseteq \sigma$ is empty, $a_{n},b_{n-1} \notin \sigma$. Since $v_{n+1}\notin V(\sigma)$, $d_{2}\notin \sigma$. Now, since $a_{n},b_{n-1},d_{2} \notin \sigma$, we have that $v_{n}\notin V(\sigma)$. Therefore, we can add $d_{2}$ to the matching $\sigma$. Hence, the matching $\sigma$ is extended to a perfect matching of $M_{2n}$. 
\end{proof}

\begin{lemma}\label{lemma:extendS}
            For $n$ even, each $s \in S$ is a bad matching of $M_{2n}$. For $n$ odd, each $s_1$ in $$S_1= \{\{b_i,c_{i+1}\},\ \{c_{j},b_{j+1}\},\ \{d_2,b_1\},\ \{d_2,c_{n-1}\}\ |\ 1\le i,j\le n-2,\ i \text{ is odd, } j \text{ is even}\}$$ is uniquely extendable to the perfect matching $ p_1=\{b_1,b_3,\dots,b_{n-2}, d_2, c_2,c_4,\dots,c_{n-1}\}$, and each $s_2$ in $$S_2=\{\{b_i,c_{i+1}\},\ \{c_{j},b_{j+1}\},\ \{d_1,c_1\}, \ \{d_1,b_{n-1}\}\ |\ 1\le i,j\le n-2,\ i \text{ is even, } j \text{ is odd}\}$$ is uniquely extendable to the perfect matching $p_2=\{b_2,b_4,\dots,b_{n-1}, c_1,c_3,\dots,c_{n-2}, d_1\}$.
\end{lemma}
\begin{proof}
  Consider the matching $\{b_i,c_{i+1}\}$, where $1\le i\le n-2$. Let us denote this matching by $\sigma_1$. Since $b_i, c_{i+1}\in \sigma_1$, $b_{i+1}, a_{i+2}$ cannot be added to $\sigma_1$, respectively. Since $b_{i+2}$ is the only remaining edge incident to $v_{i+2}$, it must be added to $\sigma_1$. It implies that $a_{i+3}$ cannot be added to $\sigma_1\cup \{b_{i+2}\}$. Since $c_{i+1}\in \sigma_1$, $c_{i+2}$ cannot be added to $\sigma_1\cup \{b_{i+2}\}$. Thus, $c_{i+3}$ must be added to the matching $\sigma_1\cup \{b_{i+2}\}$, as it the only remaining edge incident to $v_{n+i+3}$. Now, we have the matching $\{b_i,c_{i+1},b_{i+2},c_{i+3}\}$. There are several cases depending on $i$ and $n$. Let us discuss all the cases one by one. 
  
 \noindent \textbf{Case ($i,n$ both are even):} Following a similar argument as above, we finally get the matching $\{b_i,c_{i+1},b_{i+2},c_{i+3},\dots, b_{n-2}, c_{n-1}, c_{i-1}, b_{i-2},\dots, c_3,b_2,c_1\}.$
  This matching cannot be extended to a perfect matching of $M_{2n}$, since the remaining vertices $v_1$ and $v_n$ are not endpoints of any edge in $M_{2n}$. 

  \noindent \textbf{Case ($i$ is odd, $n$ is even):} Following a similar argument as above, we finally get the matching $\{b_i,c_{i+1},b_{i+2},c_{i+3},\dots, b_{n-3}, c_{n-2},b_{n-1}, c_{i-1}, b_{i-2},\dots, c_2,b_1\}.$
  This matching cannot be extended to a perfect matching of $M_{2n}$, since the remaining vertices $v_{n+1}$ and $v_{2n}$ are not endpoints of any edge in $M_{2n}$. 

  \noindent \textbf{Case ($i$ is even, $n$ is odd):} Following a similar argument as above, we get the matching $\{b_i,c_{i+1},b_{i+2},c_{i+3},\dots, b_{n-3}, c_{n-2},b_{n-1}, c_{i-1}, b_{i-2},\dots, c_3,b_2,c_1\}.$
  This matching can be extended to a perfect matching of $M_{2n}$, since the remaining vertices $v_{1}$ and $v_{2n}$ are endpoints of the edge $d_1$. Adding $d_1$, we get the perfect matching $p_2.$

  \noindent \textbf{Case ($i,n$ both are odd):} Following a similar argument as above, we get the matching $\{b_i,c_{i+1},b_{i+2},c_{i+3},\dots, b_{n-2}, c_{n-1}, c_{i-1}, b_{i-2},\dots, c_2,b_1\}.$
  This matching can be extended to a perfect matching of $M_{2n}$, since the remaining vertices $v_{n}$ and $v_{n+1}$ are endpoints of the edge $d_2$. Adding $d_2$, we get the perfect matching $p_1.$

  Similar reasoning works for the matching $\{c_i,b_{i+1}\}$ for $1\le i\le n-2$. Consider the remaining matchings of the set $S$, $\{d_1,c_{1}\},\ \{d_1,b_{n-1}\},\ \{d_2,b_{1}\},\ \{d_2,c_{n-1}\}$. Note that the reasoning for these four matchings will be similar, so we consider the matching $\{d_1,c_{1}\}$. Let us denote this matching by $\sigma_2$. Since $d_1, c_{1}\in \sigma_2$, $b_{1}, a_{2}$ cannot be added to $\sigma_2$, respectively. Since $b_{2}$ is the only remaining edge incident to $v_{2}$, it must be added to $\sigma_2$. It implies that $a_{3}$ cannot be added to $\sigma_2\cup \{b_{2}\}$. Since $c_{1}\in \sigma_2$, $c_{2}$ cannot be added to $\sigma_2\cup \{b_{2}\}$. Thus, $c_{3}$ must be added to the matching $\sigma_2\cup \{b_{2}\}$, as it the only remaining edge incident to $v_{n+3}$. Now, we have the matching $\{d_1,c_{1},b_{2},c_{3}\}$. Since $b_2,c_3$ are in this matching now, we can use the above argument given for the matching $\{b_2,c_3\}$. There are two cases depending on $n$. Let us discuss these cases.

   \noindent \textbf{Case ($n$ is even):}  We finally get the matching $\{d_1,c_{1},b_{2},c_{3},b_4,c_5,\dots, b_{n-4}, c_{n-3}, b_{n-2}\}.$ This matching cannot be extended to a perfect matching of $M_{2n}$, since the remaining vertices $v_n$ and $v_{2n-1}$ are not endpoints of any edge in $M_{2n}$. 

   \noindent \textbf{Case ($n$ is odd):}  We  get the matching $\{d_1,c_{1},b_{2},c_{3},b_4,c_5,\dots, b_{n-3}, c_{n-2}, b_{n-1}\},$ which is the perfect matching $p_2$ of $M_{2n}$. 

   Since $S_1\cup S_2=S$, we have that a matching of $S$ can be extended to the perfect matching $p_1$ or $p_2$, if $n$ is odd.
\end{proof}

\begin{lemma}\label{lemma:badmatch}
For $n$ even, $S$ is the set of all bad matchings of $M_{2n}$. For $n$ odd, a bad matching of $M_{2n}$ is either of the form $s_1 \cup \{e_1\}$ or $s_2 \cup \{e_2\}$, where $s_i\in S_i$ and $e_i\in E(M_{2n})\setminus p_i$ for $i\in \{1,2\}$.
\end{lemma}
\begin{proof}
    Since a bad matching is a minimal matching that cannot be extended to a perfect matching, the result follows from Lemmas \ref{lemma:extend} and \ref{lemma:extendS}, when $n$ is even. Now, let $n$ be odd. From Lemma \ref{lemma:extend}, we have that a bad matching must have an element from the set $S$. From Lemma \ref{lemma:extendS}, we have that a matching of $S$ can be extended to the perfect matching $p_1$ or $p_2$. This implies that a bad matching is either of the form $s_1 \cup \{e_1\}$ or $s_2 \cup \{e_2\}$, where $s_i\in S_i$ and $e_i\in E(M_{2n})\setminus p_i$ for $i\in \{1,2\}$.
\end{proof}

 Let us denote the line graph of $M_{2n}$ with additional edges $d_1c_1,\ d_1b_{n-1},\ d_2b_1, d_2c_{n-1},\ b_ic_{i+1}$ and $c_ib_{i+1}\  (1\le i\le n-2)$ by $G_n$ (Figure \ref{fig:G}). The vertex set of $G_n$, $V(G_n)=E(M_{2n})$ and the edge set $$E(G_n)=\{uv\ |\ u \text{ and } v \text{ are adjacent edges in } M_{2n}\} \ \cup$$ $$ \{d_1c_1,\ d_1b_{n-1},\ d_2b_1,\ d_2c_{n-1},\ b_ic_{i+1},\ c_ib_{i+1}\ |\ 1\le i\le n-2\}.$$

\begin{figure}[ht]
\tikzstyle{vert}=[circle, draw, fill=black!100, inner sep=0pt, minimum width=3pt]
\tikzstyle{ver}=[]
\tikzstyle{extra}=[circle, draw, fill=black!50, inner sep=0pt, minimum width=4pt]
\tikzstyle{edge} = [draw,thick,-]
\centering
\begin{tikzpicture}[scale=1]

\foreach \x/\y/\z in {7/-1/12,7/1/13,5/-1/0,5/1/1,3/-1/2,3/1/3,1/-1/4,1/1/5,-1/-1/6,-1/1/7,-3/-1/8,-3/1/9,-5/-1/10,-5/1/11, -4/0/14, -2/0/15, 0/0/16, 4/0/17, 6/0/18}
{\node[vert] (a\z) at (\x,\y){};}

\foreach \x/\y in {2.5/-1,2/-1,1.5/-1,2.5/1,2/1,1.5/1}{\node[extra] () at (\x,\y){};}

\foreach \x/\y in {a0/a2,a6/a4,a8/a6,a1/a3,a7/a5,a9/a7,a8/a10,a9/a11, a0/a12, a1/a13, a14/a8, a14/a9, a14/a10, a14/a11, a15/a6, a15/a7, a15/a8, a15/a9, a16/a4, a16/a5, a16/a6, a16/a7, a17/a0, a17/a1, a17/a2, a17/a3, a18/a0, a18/a1, a18/a12, a18/a13}
{\draw [edge]  (\x) -- (\y);}

\foreach \x/\y/\z in {7.25/-1.4/d_2, 7.25/1.4/d_1, -4.4/0/a_1, -2.4/0/a_2, -0.4/0/a_3, 4.6/0/a_{n-1}, 6.5/0/a_n  }
{\node[ver] () at (\x,\y){\small{$ \z$}};}

\foreach \x/\y/\z/\w in {a8/a11/-3.9/0.25, a6/a9/-1.9/0.25, a4/a7/0.1/0.25, a0/a3/3.9/-0.25, a12/a1/5.9/-0.25, a10/a9/-3.9/-0.25, a8/a7/-1.9/-0.25, a5/a6/0.1/-0.25, a1/a2/3.9/0.25, a0/a13/5.9/0.25 }
{\draw [edge] plot [smooth,tension=1.5] coordinates{(\x) (\z,\w) (\y)};}

\node[ver] () at (5,-1.4){\small{$ c_{n-1}$}};
\node[ver] () at (5,1.4){\small{$ b_{n-1}$}};
\node[ver] () at (3,-1.4){\small{$ c_{n-2}$}};
\node[ver] () at (3,1.4){\small{$ b_{n-2}$}};
\node[ver] () at (1,-1.4){\small{$ c_{3}$}};
\node[ver] () at (1,1.4){\small{$ b_3$}};
\node[ver] () at (-1,-1.4){\small{$ c_{2}$}};
\node[ver] () at (-1,1.4){\small{$ b_2$}};
\node[ver] () at (-3,-1.4){\small{$ c_{1}$}};
\node[ver] () at (-3,1.4){\small{$ b_1$}};
\node[ver] () at (-5.25,1.4){\small{$ d_1$}};
\node[ver] () at (-5.25,-1.4){\small{$ d_{2}$}};

\end{tikzpicture}
\caption{$G_n$}\label{fig:G}
\end{figure}

Let $H_n$ be a graph (Figure \ref{fig:H}) with the vertex set $$V(H_n)=\{a_i,\ b_j \ |\ 1\le i\le n, 1\le j\le n-1\}$$ and the edge set $$E(H_n)=\{a_1b_1,\ a_nb_{n-1},\ b_ib_{i+1},\  a_jb_{j-1},\ a_jb_j\ |\ 1\le i\le n-2,\ 2\le j\le n-1\}.$$

\begin{figure}[ht]
\tikzstyle{vert}=[circle, draw, fill=black!100, inner sep=0pt, minimum width=3pt]
\tikzstyle{ver}=[]
\tikzstyle{extra}=[circle, draw, fill=black!50, inner sep=0pt, minimum width=4pt]
\tikzstyle{edge} = [draw,thick,-]
\centering
\begin{tikzpicture}[scale=1]

\foreach \x/\y/\z in {5/1/b5,3/1/b4,1/1/b3,-1/1/b2,-3/1/b1, -4/0/a1, -2/0/a2, 0/0/a3, 4/0/a4, 6/0/a5, 8/0/a6, 7/1/b6}
{\node[vert] (\z) at (\x,\y){};}

\foreach \x/\y in {2.5/1,2/1,1.5/1}{\node[extra] () at (\x,\y){};}

\foreach \x/\y in {b5/b4,b2/b3,b1/b2, a1/b1,  a2/b2, a2/b1, a3/b3, a3/b2, a4/b4, a4/b5, a5/b5, a5/b6, b6/b5, b6/a6}
{\draw [edge]  (\x) -- (\y);}

\foreach \x/\y/\z in {  -4.4/0/a_1, -2.4/0/a_2, -0.4/0/a_3, 4.6/0/a_{n-2}, 6.6/0/a_{n-1} , 8.4/0/a_n, 7/1.4/b_{n-1} }
{\node[ver] () at (\x,\y){\small{$ \z$}};}

\node[ver] () at (5,1.4){\small{$ b_{n-2}$}};

\node[ver] () at (3,1.4){\small{$ b_{n-3}$}};

\node[ver] () at (1,1.4){\small{$ b_3$}};

\node[ver] () at (-1,1.4){\small{$ b_2$}};

\node[ver] () at (-3,1.4){\small{$ b_1$}};

\end{tikzpicture}
\caption{$H_n$}\label{fig:H}
\end{figure}

\begin{lemma}\label{lemma:H}
    For $n\ge 2$, the Independence complex $$ \operatorname{ Ind}(H_n)\simeq \begin{cases}
       \mathbb{S}^{\frac{n-2}{2}} & \text{if } n \text{ is even}, \\ \ast & \text{if } n\text{ is odd}.
   \end{cases}$$
\end{lemma}
\begin{proof}
  Since $N(a_n)\subset N(b_{n-2})$, we have that Ind$(H_n)\simeq$ Ind$(H_n\setminus b_{n-2})$ by Proposition \ref{easy}. Observe that the graph $H_n\setminus b_{n-2}$ is $H_{n-2} \sqcup P_3$. Thus, Ind$(H_n)\simeq$ Ind$(H_{n-2})\ \star$ Ind$(P_3)$. Note that $H_2$ is $P_3$ and $H_1$ is $P_1$. Applying Proposition \ref{easy} iteratively, we finally get that Ind$(H_n)\simeq \star_{n/2}$ Ind$(P_3)$ for $n$ even, and Ind$(H_n)\simeq a_1 \star (\star_{(n-1)/2}$ Ind$(P_3))\simeq \ast$ for $n$ odd. Due to Proposition \ref{path}, we have that Ind$(H_n)\simeq \star_{n/2} \mathbb{S}^0=\mathbb{S}^{\frac{n-2}{2}}$ for $n$ even.
\end{proof}

\begin{lemma}\label{lemma:G}
    For $n\ge 2$, the Independence complex $$ \operatorname{ Ind}(G_n)\simeq \begin{cases}
       \mathbb{S}^{\frac{n-2}{2}}\ \vee \ \mathbb{S}^{\frac{n-2}{2}} & \text{if } n \text{ is even}, \\ \ast & \text{if } n\text{ is odd}.
   \end{cases}$$

\end{lemma}

\begin{proof}
     Observe that $N(d_1)=N(d_2)$ and $N(c_i)=N(b_i)$ for all $1\le i\le n-1$. Applying Proposition \ref{easy} on $G_n$ w.r.t. $d_1$ and $d_2$, we have that Ind$(G_n)\simeq $ Ind$(G_n\setminus d_2)$. Applying Proposition \ref{easy} repeatedly, we finally get that Ind$(G_n)\simeq $ Ind$(G_n\setminus \{d_2,\ c_i\ |\ 1\le i\le n-1\})$. Let us denote the graph $G_n\setminus \{d_2,\ c_i\ |\ 1\le i\le n-1\}$ by $G_n^\prime$. Thus, we have that Ind$(G_n)\simeq $ Ind$(G_n^\prime)$.

    Consider the graph $G_n^\prime\setminus d_1$ and $G_n^\prime\setminus N[d_1]$. The graph $G_n^\prime\setminus d_1$ is the same as the graph $H_n$, and the graph $G_n^\prime\setminus N[d_1]$ is the same as the graph $H_{n-2}$. For $n$ even, Ind$(G_n^\prime\setminus d_1) \simeq \mathbb{S}^{\frac{n-2}{2}}$ and Ind$(G_n^\prime\setminus N[d_1]) \simeq \mathbb{S}^{\frac{n-4}{2}}$ due to Lemma \ref{lemma:H}. For $n$ odd, Ind$(G_n^\prime\setminus d_1) \simeq$ Ind$(G_n^\prime\setminus N[d_1]) \simeq \ast$. Therefore, we can apply Proposition \ref{ld} on the graph $G_n^\prime$ w.r.t. $d_1$. Applying this Proposition, the result follows.
\end{proof}

\begin{theorem}\label{PMC}
    For $n$ even, the Perfect Matching complex $\mathcal M_p(M_{2n}) \simeq \mathbb{S}^{\frac{n-2}{2}}\ \vee \ \mathbb{S}^{\frac{n-2}{2}}.$ For $n$ odd, 
    $$ \mathcal M_p(M_{2n})\simeq \begin{cases}
       \bigvee_4 \ \mathbb{S}^{2k+1} & \text{if } n=3(2k+1), \\ \bigvee_2 \ \mathbb{S}^{2k+2} & \text{if } n=3(2k+1)+2, 3(2k+1)+4.
   \end{cases}$$
\end{theorem}

\begin{proof}
    Due to Lemma \ref{lemma:badmatch}, $\mathcal M_p(M_{2n})=$ Ind$(G_n)$ for $n$ even. Thus, $\mathcal M_p(M_{2n}) \simeq \mathbb{S}^{\frac{n-2}{2}}\ \vee \ \mathbb{S}^{\frac{n-2}{2}}$ by Lemma \ref{lemma:G}, if $n$ is even. Now, let us consider the case when $n$ is odd. Due to Lemma \ref{lemma:badmatch}, $\mathcal M_p(M_{2n})=$ Ind$(G_n) \cup p_1 \cup p_2$, where $p_1$ and $p_2$ are perfect matchings $\{b_1,b_3,\dots,b_{n-2}, d_2,c_2,c_4,\\ \dots,c_{n-1}\}$ and $\{b_2,b_4,\dots,b_{n-1}, c_1,c_3,\dots,c_{n-2}, d_1\}$, respectively. Since Ind$(G_n)\simeq \ast$ by Lemma \ref{lemma:G}, and $p_1 \cap p_2$ is empty, $$\mathcal M_p(M_{2n})= \bigvee_{i=1}^{2} \Sigma (\operatorname{Ind}(G_n)\cap p_i)$$ by Proposition \ref{prop:bjorner} $(ii)$. By the definition of $G_n$ and the fact that $p_i$ is a simplex we have that $\operatorname{Ind}(G_n)\cap p_i= \operatorname{Ind}(C_n^i)$, where $C_n^1$ is the $n$-cycle $b_1c_2b_3c_4\dots b_{n-2}c_{n-1}d_2b_1$ and $C_n^2$ is the $n$-cycle $c_1b_2c_3b_4\dots c_{n-2}b_{n-1}d_1c_1$. Now, the result follows from Proposition \ref{cycle}.
\end{proof}

\subsection{Perfect matching complex of circular ladder graph}

Let us now study the homotopy type of the perfect matching complex of the circular ladder graph. 

\begin{lemma}\label{lemma:extendCL}
Let $S$ be the set $$\{\{b_i,c_{i+1}\},\ \{c_i,b_{i+1}\},\ \{b_n,c_1\},\ \{c_n,b_1\} \mid 1\le i\le n-1 \}.$$ If $\sigma$ is a matching of $\mathcal C_{2n}$ such that $s \nsubseteq \sigma$  for all $s\in S$, then $\sigma$ can be extended to a perfect matching of $\mathcal C_{2n}$ for all $n\ge 2$.
\end{lemma}

\begin{proof}
  Let $\sigma$ be a matching of $\mathcal C_{2n}$ such that $s \nsubseteq \sigma$ for all $s\in S$. Let $V(\sigma)\subset V(\mathcal C_{2n})$ be the set of all vertices of $\mathcal C_{2n}$ that are endpoints of the edges of $\sigma$. Consider the edge $a_j\in E(\mathcal C_{2n})$ for $1\le j\le n$. For notation, see Figure \ref{fig:CL}. The endpoints $v_j$ and $v_{n+j}$ of $a_j$ satisfy exactly one of the following four conditions:
\begin{enumerate}[$(i)$]
    \item $v_j, v_{n+j} \in V(\sigma)$;
    \item $v_j \in V(\sigma)$ and $v_{n+j} \notin V(\sigma)$;
    \item $v_j \notin V(\sigma)$ and $v_{n+j} \in V(\sigma)$;
    \item $v_j, v_{n+j} \notin V(\sigma)$.
\end{enumerate}

We want to extend the matching $\sigma$ to a perfect matching. Therefore, when condition $(i)$ arises, look for some other $a_j$ that satisfies conditions $(ii)$, $(iii)$ or $(iv)$. If $a_j$ satisfies condition $(iv)$, then add $a_j$ to the matching $\sigma$. Note that conditions $(ii)$ and $(iii)$ are similar, so we consider condition $(ii)$. 

Let $a_j$ satisfy the condition $(ii)$. This implies that $a_j\notin \sigma$. Since $a_j\notin \sigma$ and $v_j\in V(\sigma)$, either $b_{j-1}\in \sigma$ ($b_n\in \sigma$ if $j=1$) or $b_{j}\in \sigma$. Let us first assume that $b_{j}\in \sigma$. This implies that $a_{j+1}\notin \sigma$ ($a_{1}\notin \sigma$ if $j=n$). Since $v_{n+j}\notin V(\sigma)$, $c_j\notin \sigma$. Since $b_{j}\in \sigma$ and $\{b_{j},c_{j+1}\} \nsubseteq \sigma $ ($\{b_{n},c_{1}\} \nsubseteq \sigma $ if $j=n$), $c_{j+1}\notin \sigma$ ($c_{1}\notin \sigma$ if $j=n$). Now, since $c_j,c_{j+1},a_{j+1} \notin \sigma$ ($c_n,c_{1},a_{1} \notin \sigma$ if $j=n$), we have that $v_{n+j+1}\notin V(\sigma)$ ($v_{n+1}\notin V(\sigma)$ if $j=n$). Therefore, we can add $c_j$ to the matching $\sigma$.

Let us now assume that $b_{j-1}\in \sigma$ ($b_{j-1}\in \sigma$ if $j=1$). Since $b_{j-1}\in \sigma$ and $\{c_{j-2},b_{j-1}\} \nsubseteq \sigma  $, $a_{j-1},c_{j-2} \notin \sigma$ ($a_{n},c_{n-1} \notin \sigma$ if $j=1$). Since $v_{n+j}\notin V(\sigma)$, $c_{j-1}\notin \sigma$ ($c_{n}\notin \sigma$ if $j=1$). Now, since $c_{j-1},c_{j-2},a_{j-1} \notin \sigma$ ($c_{n},c_{n-1},a_{n} \notin \sigma$ if $j=1$), we have that $v_{n+j-1}\notin V(\sigma)$ ($v_{2n}\notin V(\sigma)$ if $j=1$). Therefore, we can add $c_{j-1}$ ($c_{n}$ if $j=1$) to the matching $\sigma$. Hence, the matching $\sigma$ is extended to a perfect matching of $\mathcal C_{2n}$. 
\end{proof}

\begin{lemma}\label{lemma:extendSCL}
            For $n$ odd, each $s \in S$ is a bad matching of $\mathcal C_{2n}$. For $n$ even, each $s_1$ in $$S_1= \{\{b_i,c_{i+1}\},\ \{c_{j},b_{j+1}\},\ \{c_n,b_1\} \mid  1\le i,j\le n-1,\ i \text{ is odd, } j \text{ is even}\}$$ is uniquely extendable to the perfect matching $ p_1=\{b_1,b_3,\dots,b_{n-1}, c_n, c_2,c_4,\dots,c_{n-2}\}$, and each $s_2$ in $$S_2=\{\{b_i,c_{i+1}\},\ \{c_{j},b_{j+1}\},\ \{b_n,c_1\}, \mid 1\le i,j\le n-1,\ i \text{ is even, } j \text{ is odd}\}$$ is uniquely extendable to the perfect matching $p_2=\{b_2,b_4,\dots,b_{n}, c_1,c_3,\dots,c_{n-1}\}$.
\end{lemma}

\begin{proof}
   Since $\mathcal C_{2n}$ is a circular ladder, the argument for the matchings $\{b_i,c_{i+1}\}$ and $\{b_n,c_{1}\}$ is similar, for $1\le i\le n-1$. Let us consider the matching $\{b_i,c_{i+1}\}$, where $1\le i\le n-1$. Let us denote this matching by $\sigma$. Since $b_i, c_{i+1}\in \sigma$, $b_{i+1}, a_{i+2}$ cannot be added to $\sigma$, respectively. Since $b_{i+2}$ is the only remaining edge incident to $v_{i+2}$, it must be added to $\sigma$. This implies that $a_{i+3}$ cannot be added to $\sigma \cup \{b_{i+2}\}$. Since $c_{i+1}\in \sigma$, $c_{i+2}$ cannot be added to $\sigma \cup \{b_{i+2}\}$. Thus, $c_{i+3}$ must be added to the matching $\sigma \cup \{b_{i+2}\}$, as this is the only edge remaining incident to $v_{n+i+3}$. Now, we have the matching $\{b_i,c_{i+1},b_{i+2},c_{i+3}\}$. There are several cases depending on $i$ and $n$. Let us discuss all the cases one by one. 
  
 \noindent \textbf{Case ($i,n$ both are even):} Following a similar argument as above, we finally get the matching $\{b_i,c_{i+1},b_{i+2},c_{i+3},\dots, b_{n-2}, c_{n-1}, c_{i-1}, b_{i-2},\dots, c_3,b_2,c_1\}.$ This matching can be extended to a perfect matching of $\mathcal C_{2n}$, since the remaining vertices $v_{1}$ and $v_{n}$ are endpoints of the edge $b_n$. Adding $b_n$, we get the perfect matching $p_2.$

  \noindent \textbf{Case ($i$ is odd, $n$ is even):} Following a similar argument as above, we finally get the matching $\{b_i,c_{i+1},b_{i+2},c_{i+3},\dots, b_{n-3}, c_{n-2},b_{n-1}, c_{i-1}, b_{i-2},\dots, c_2,b_1\}.$
  This matching can be extended to a perfect matching of $\mathcal C_{2n}$, since the remaining vertices $v_{n+1}$ and $v_{2n}$ are endpoints of the edge $c_n$. Adding $c_n$, we get the perfect matching $p_1.$

  \noindent \textbf{Case ($i$ is even, $n$ is odd):} Following a similar argument as above, we get the matching $\{b_i,c_{i+1},b_{i+2},c_{i+3},\dots, b_{n-3}, c_{n-2},b_{n-1}, c_{i-1}, b_{i-2},\dots, c_3,b_2,c_1\}.$
This matching cannot be extended to a perfect matching of $\mathcal C_{2n}$, since the remaining vertices $v_{1}$ and $v_{2n}$ are not endpoints of any edge in $\mathcal C_{2n}$.

  \noindent \textbf{Case ($i,n$ both are odd):} Following a similar argument as above, we get the matching $\{b_i,c_{i+1},b_{i+2},c_{i+3},\dots, b_{n-2}, c_{n-1}, c_{i-1}, b_{i-2},\dots, c_2,b_1\}.$ This matching cannot be extended to a perfect matching of $\mathcal C_{2n}$, since the remaining vertices $v_{n}$ and $v_{n+1}$ are not endpoints of any edge in $\mathcal C_{2n}$.

  Similar reasoning works for the matchings $\{c_i,b_{i+1}\}$ and $\{c_6,b_1\}$ for $1\le i\le n-1$. Hence, the result follows.
\end{proof}

\begin{lemma}\label{lemma:badmatchCL}
For $n$ odd, $S$ is the set of all bad matchings of $\mathcal C_{2n}$. For $n$ even, a bad matching of $\mathcal C_{2n}$ is either of the form $s_1 \cup \{e_1\}$ or $s_2 \cup \{e_2\}$, where $s_i\in S_i$ and $e_i\in E(\mathcal C_{2n})\setminus p_i$ for $i\in \{1,2\}$.
\end{lemma}

\begin{proof}
 Since a bad matching is a minimal matching that cannot be extended to a perfect matching, the result follows from Lemmas \ref{lemma:extendCL} and \ref{lemma:extendSCL}, when $n$ is odd. Now, let $n$ be even. From Lemma \ref{lemma:extendCL}, we have that a bad matching must have an element from the set $S$. From Lemma \ref{lemma:extendSCL}, we have that a matching of $S$ can be extended to the perfect matching $p_1$ or $p_2$. This implies that a bad matching is of the form $s_1 \cup \{e_1\}$ or $s_2 \cup \{e_2\}$, where $s_i\in S_i$ and $e_i\in E(\mathcal C_{2n})\setminus p_i$ for $i\in \{1,2\}$.   
\end{proof}

Let us denote the line graph of $\mathcal C_{2n}$ with additional edges $b_nc_1,\ c_nb_1,\ b_ic_{i+1}$ and $c_ib_{i+1}\  (1\le i\le n-1)$ by $K_n$. The vertex set of $K_n$, $V(K_n)=E(\mathcal C_{2n})$ and the edge set $$E(K_n)=\{uv\ |\ u \text{ and } v \text{ are adjacent edges in } \mathcal C_{2n}\} \ \cup\  \{b_nc_1,\ c_nb_1,\ b_ic_{i+1},\ c_ib_{i+1} \mid 1\le i\le n-1\}.$$ Note that $K_n$ is isomorphic to the graph $G_n$ in Figure \ref{fig:G}. If the vertices $b_n$ and $c_n$ of $K_n$ are renamed $d_1$ and $d_2$, respectively, the graph $G_n$ is obtained.

\begin{theorem}\label{PMCCL}
    For $n$ odd, the perfect matching complex $\mathcal M_p(\mathcal C_{2n}) \simeq \ast.$ For $n$ even, 
    $$ \mathcal M_p(\mathcal C_{2n})\simeq \begin{cases}
       \mathbb{S}^{\frac{n-2}{2}}\ \vee \ \mathbb{S}^{\frac{n-2}{2}} \bigvee_4 \ \mathbb{S}^{2k} & \text{if } n=3(2k), \\ \mathbb{S}^{\frac{n-2}{2}}\ \vee \ \mathbb{S}^{\frac{n-2}{2}} \bigvee_2 \ \mathbb{S}^{2k+1} & \text{if } n=3(2k+1)\pm 1.
   \end{cases}$$
\end{theorem}

\begin{proof}
      Due to Lemma \ref{lemma:badmatchCL}, $\mathcal M_p(\mathcal C_{2n})=$ Ind$(K_n)$ for $n$ odd. Thus, $\mathcal M_p(\mathcal C_{2n}) \simeq \ast$ by Lemma \ref{lemma:G}, if $n$ is odd. Now, let us consider the case where $n$ is even. Due to Lemma \ref{lemma:badmatchCL}, $\mathcal M_p(\mathcal C_{2n})=$ Ind$(K_n) \cup p_1 \cup p_2$, where $p_1$ and $p_2$ are perfect matchings $\{b_1,b_3,\dots,b_{n-1}, c_n,c_2,c_4, \dots,c_{n-2}\}$ and $\{b_2,b_4,\dots,b_{n}, c_1,c_3,\dots,c_{n-1}\}$, respectively. Since $p_1 \cap p_2$ is empty, 
      $$\mathcal M_p(\mathcal C_{2n})= \operatorname{Ind}(K_n) \cup \bigcup_{1\le i\le 2} \text{cone} (\operatorname{Ind}(K_n)\cap p_i) $$
      by Proposition \ref{prop:bjorner} $(i)$. By the definition of $K_n$ and the fact that $p_i$ is a simplex, we have $\operatorname{Ind}(K_n)\cap p_i= \operatorname{Ind}(C_n^i)$, where $C_n^1$ is the $n$-cycle $b_1c_2b_3c_4\dots b_{n-1}c_{n}b_1$ and $C_n^2$ is the $n$-cycle $c_1b_2c_3b_4\dots c_{n-1}b_{n}c_1$. It follows from Proposition \ref{cycle}, for $n=3(2k)$, $ \operatorname{ Ind}(C_n^i)\simeq  \mathbb{S}^{2k-1} \vee \mathbb{S}^{2k-1}$. By Lemma \ref{lemma:G}, $\operatorname{Ind}(K_n)\cong \operatorname{Ind}(G_n)\simeq \mathbb{S}^{\frac{n-2}{2}}\ \vee \ \mathbb{S}^{\frac{n-2}{2}}$.

Since $\operatorname{Ind}(C_n^i)\simeq \mathbb{S}^{2k-1}\vee\mathbb{S}^{2k-1}$ and $\operatorname{Ind}(K_n)\simeq \mathbb{S}^{\frac{n-2}{2}}\vee\mathbb{S}^{\frac{n-2}{2}}$, the subcomplex $\operatorname{Ind}(C_n^i)$ is contractible in $\operatorname{Ind}(K_n)$ when $n=3(2k)$. Similarly, when $n=3(2k+1)\pm1$, we have $\operatorname{Ind}(C_n^i)\simeq\mathbb{S}^{2k}$ and $\operatorname{Ind}(K_n)\simeq \mathbb{S}^{\frac{n-2}{2}}\vee\mathbb{S}^{\frac{n-2}{2}}$, and hence $\operatorname{Ind}(C_n^i)$ is contractible in $\operatorname{Ind}(K_n)$. Therefore, by Proposition \ref{cone and suspension}, we have
$$\mathcal M_p(\mathcal C_{2n})\simeq\begin{cases}
       \mathbb{S}^{\frac{n-2}{2}}\ \vee \ \mathbb{S}^{\frac{n-2}{2}} \bigvee_2 \Sigma (\mathbb{S}^{2k-1} \vee \mathbb{S}^{2k-1}) & \text{if } n=3(2k), \\ \mathbb{S}^{\frac{n-2}{2}}\ \vee \ \mathbb{S}^{\frac{n-2}{2}} \bigvee_2 \ \Sigma(\mathbb{S}^{2k}) & \text{if } n=3(2k+1)\pm 1.
   \end{cases}$$
Therefore,
$$\mathcal M_p(\mathcal C_{2n})\simeq\begin{cases}
       \mathbb{S}^{\frac{n-2}{2}}\ \vee \ \mathbb{S}^{\frac{n-2}{2}} \bigvee_4 \ \mathbb{S}^{2k} & \text{if } n=3(2k), \\ \mathbb{S}^{\frac{n-2}{2}}\ \vee \ \mathbb{S}^{\frac{n-2}{2}} \bigvee_2 \ \mathbb{S}^{2k+1} & \text{if } n=3(2k+1)\pm 1.
   \end{cases}$$
This completes the proof.
  \end{proof}

\section{Future Directions}
The explicit determination of the homotopy types of independence complexes and perfect matching complexes of M\"obius and circular ladder graphs naturally leads to several further questions.

\begin{problem}
{\rm
A natural direction for future research is to investigate other graph-associated simplicial complexes arising from M\"obius ladder graphs $M_{2n}$ and circular ladder graphs $\mathcal{C}_{2n}$. In addition to the independence and perfect matching complexes, one may consider, for example, the matching complex, neighborhood complex, clique complex, cut complex, and total cut complex associated with these graphs. These complexes encode different combinatorial structures and may exhibit rich and distinct topological behavior. It would be interesting to determine their homotopy types, homology groups, and other topological invariants, and to investigate whether the symmetries of $M_{2n}$ and $\mathcal{C}_{2n}$ lead to periodic or otherwise systematic patterns in the topology of the associated complexes.

}
\end{problem}

\begin{problem}
{\rm
Another interesting problem is to investigate the topology of the independence complexes and perfect matching complexes for other highly symmetric graph families related to the M\"obius ladder graphs $M_{2n}$ and the circular ladder graphs $\mathcal{C}_{2n}$. In particular, it would be interesting to study suitable $4$-regular graph families and determine whether their associated complexes exhibit periodic homotopy types, analogous combinatorial patterns, or other systematic topological behavior.
}
\end{problem}

\bigskip

\noindent {\bf Acknowledgement:}  The first author is supported by the Institute fellowship from the Indian Institute of Technology Delhi, India.

\smallskip

\noindent {\bf Data availability:} The authors declare that all data supporting the findings of this study are available within the article.

\smallskip

\noindent {\bf Declarations}

\noindent {\bf Conflict of interest:} No potential conflict of interest was reported by the authors.

\end{document}